\documentclass[a4,11pt]{article}
	\usepackage{amsthm}
	\usepackage{amsmath}
	\usepackage{amssymb}
	\usepackage[T1]{fontenc}
	\usepackage{enumerate}
	\usepackage{bbm}
	\usepackage{url}
\renewcommand{\mathbb}{\mathbbm}                     
	
\renewcommand{\epsilon}{\varepsilon}                 
\renewcommand{\phi}{\varphi}
\renewcommand{\le}{\leqslant}
\renewcommand{\ge}{\geqslant}
\renewcommand{\leq}{\le}
\renewcommand{\geq}{\ge}

\newcommand{\abs}[1]{\left\lvert #1 \right\rvert}    

\DeclareMathOperator{\C}{{\mathbb  C}}                
\DeclareMathOperator{\R}{{\mathbb R}}                
\DeclareMathOperator{\N}{{\mathbb N}}                
\DeclareMathOperator{\Borel}{{\mathfrak B}}
\DeclareMathOperator{\I}{Id}

\newcommand{\scapro}[2]{\langle #1,#2\rangle}       

\allowdisplaybreaks

\numberwithin{equation}{section}

\theoremstyle{plain}
\newtheorem{thm}{\protect\theoremname}[section]		
\theoremstyle{definition}
\newtheorem{defn}[thm]{\protect\definitionname}
\theoremstyle{plain}
\theoremstyle{definition}
\newtheorem{example}[thm]{Example}
\newtheorem{lemma}[thm]{Lemma}
\newtheorem{remark}[thm]{Remark}
 \providecommand{\definitionname}{Definition}
\providecommand{\theoremname}{Theorem}	
	
\newcommand{\ud}{\,\mathrm{d}}
\newcommand{\udd}{\mathrm{d}}

\newcommand{\1}{\mathbbm{1}}

\newcommand\numberthis{\addtocounter{equation}{1}\tag{\theequation}}
	
\newcommand{\itemEq}[1]{%
	\begingroup%
   	\setlength{\abovedisplayskip}{0pt}%
	\setlength{\belowdisplayskip}{0pt}%
	\parbox[c]{\linewidth}{\begin{flalign}#1&&\end{flalign}}%
	\endgroup}
	
	\newcommand{\beq}{\begin{equation}}
	\newcommand{\beql}[1]{\begin{equation}\label{#1}}
	\newcommand{\eeq}{\end{equation}}

\title{The stochastic Cauchy problem driven by a cylindrical L\'evy process}

\author{Umesh Kumar \\Department of Mathematics \\ King's College\\ London WC2R 2LS\\ United Kingdom\\ \\
nouf.umesh@kcl.ac.uk \and Markus Riedle \\Department of Mathematics \\ King's College  \\ London WC2R 2LS\\ United Kingdom\\ \\ markus.riedle@kcl.ac.uk}

\begin{document}

	\maketitle
	
	\begin{abstract}
In this work, we derive sufficient and necessary conditions for the existence of a weak and mild solution of an abstract stochastic Cauchy problem driven by an arbitrary cylindrical L\'evy process. Our approach requires to establish a stochastic Fubini result for stochastic integrals with respect to cylindrical L\'evy processes. This approach enables us to conclude that the solution process has almost surely scalarly square integrable paths. Further properties of the solution such as the Markov property and stochastic continuity are derived.
	\end{abstract}
\section{Introduction}
Cylindrical L\'evy processes naturally extend the class of cylindrical Brownian motions, which have been the standard model for random perturbations of partial differential equations for the last 50 years. The general concept of cylindrical L\'evy processes in Banach spaces has been recently introduced by Applebaum and Riedle in \cite{app}. However, so far only specific examples of 
cylindrical L\'evy processes have been applied for modelling the driving noise of stochastic partial differential equations.

In this work we consider a linear evolution equation driven by an additive noise, 
or equivalently a stochastic Cauchy problem, of the form:
\begin{equation}\label{SCP10}
\ud Y(t)= AY(t)\ud t+B\ud L(t)  \qquad  \text{for all $t \in [0,T]$.} 
\end{equation}
Here, $L$ is a cylindrical L\'evy process on a separable Hilbert space $U$,
the coefficient $A$ is the generator of a strongly continuous semigroup on a separable Hilbert space $V$ and $B$  is a linear, bounded operator from $U$ to $V$. In this general setting, we 
present a complete theory for the existence of a mild and weak solution of \eqref{SCP10} and derive some fundamental properties of the solution and its trajectories. 

Only for specific examples of cylindrical L\'evy processes $L$ and sometimes under further restrictive assumptions on the generator $A$, the stochastic Cauchy problem \eqref{SCP10} has been considered in most of the literature. There, typically  one of the following two approaches are exploited: either the considered cylindrical L\'evy process $L$ is of such a form that the question of existence of a weak solution reduces to the study of a sequence of infinitely many one-dimensional 
processes or the underlying Hilbert space $V$ is embedded in a larger space. 
The first approach is applied for example in the works Brze\'zniak et al \cite{Brzetal}, Liu and Zhai \cite{LiuZhai},  and  Priola and Zabczyk \cite{priola_zabczyk}. In these publications, the considered examples of a cylindrical L\'evy process $L$ only act along the eigenbasis of the generator $A$ in an independent way. The second approach is utilised for example in the works  \cite{brzezniak_neerven} by Brze\'zniak and
van Neerven for cylindrical Brownian motion,  \cite{brzezniak_zabczyk} by Brze\'zniak and Zabczyk for a cylindrical L\'evy process modelled as a subordinated cylindrical Brownian motion or \cite{PeszatZab12} by Peszat and  Zabczyk for a general case. Although this approach is elegant and natural, one typically obtains conditions for the existence of a weak solution in terms of the larger space which per se is not related to the equation under consideration. 

The stochastic Cauchy problem \eqref{SCP10} exhibits a new phenomena which has not been observed in the Gaussian setting, i.e.\ when $L$ is a cylindrical Brownian motion: the solution may exist as a stochastic process in the underlying Hilbert space $V$, but its trajectories are highly irregular; see for example Brze\'zniak et al \cite{Brzetal}, 
Brze\'zniak and Zabczyk \cite{brzezniak_zabczyk} and Peszat and Zabczyk \cite{peszat_zabczyk}. In fact, the only positive results on some analytical regularity of the paths can be found in  Liu and Zhai \cite{LiuZhai16} and Peszat and Zabczyk \cite{PeszatZab12}. However, these results are very restrictive and do not cover most of the considered examples of cylindrical L\'evy processes.

For establishing the existence of a weak solution, the general noise considered in our work prevents us from following standard arguments as exploited for genuine L\'evy processes,  attaining values in $V$. In this case,  one can either utilise the L\'evy-It{\^o} decomposition
together with a stochastic Fubini theorem for the martingale part as it is done by Applebaum in \cite{app_martingale} or by Peszat and Zabczyk in \cite{peszat_zabczyk}, or an integration by parts formula as accomplished by  Chojnowska-Michalik in \cite{michalik}. However, since our noise is cylindrical it does not enjoy a 
 L\'evy-It{\^o} decomposition in the underlying Hilbert space. Also exploiting an integration by parts  formula seems to be excluded as such a formula would indicate certain regularity  of the paths. We circumvent these problems by applying a stochastic Fubini theorem but without decomposing the integrator of the stochastic integral. 
 
However, also the stochastic Fubini theorem  cannot be derived by standard methods due to the lack of a   L\'evy-It{\^o} decomposition of the cylindrical L\'evy process in the underlying Hilbert space.   Even more, most of the results require finite moments of the stochastic integral, which is not guaranteed in our general framework; see Applebaum \cite{app_martingale}, Da Prato and Zabczyk \cite{daprato_zab} and Filipovi{\'c} et al \cite{teichman}.
In our work, we succeed in establishing a stochastic Fubini result by using the observation that the iterated integrals can be considered as the inner product in 
a space of integrable functions. This  observation and its elegant utilisation originates from the work van Neerven and Veraar  \cite{neerven_veraar}. Similar as in this work \cite{neerven_veraar}, however without having the theory of 
$\gamma$-radonifying operators at hand, we derive by tightness arguments, that the parameterised stochastic integral with respect to a cylindrical L\'evy process defines a random variable in a space of integrable functions, which enables us to consider the iterated integrals as an inner product. 
 
Surprisingly, our stochastic Fubini result and its application to the representation of the weak solution of \eqref{SCP10} immediately yields that the trajectories of the solution are scalarly square integrable. As far as we know, this is the first positive result on an analytical path property of the solution of the stochastic Cauchy problem independent of the driving cylindrical L\'evy process. Furthermore, having established the representation of the solution for \eqref{SCP10} by a stochastic integral, which itself is based on the rich theory of 
cylindrical measures and cylindrical random variables, enables us to study further properties of the solution and its trajectories. More specifically, we show without any assumptions on the cylindrical L\'evy process that the solution process is a Markov process and continuous in probability. For specific examples of cylindrical L\'evy processes, 
these properties were established in \cite{brzezniak_zabczyk} and \cite{priola_zabczyk}. However, there the arguments are strongly restricted to the specific examples under consideration.  We are also able to provide a condition in our general framework which implies the non-existence of a modification of the solution with scalarly c{\`a}dl{\`a}g trajectories,   a phenomena which has often been observed in several publications above cited. In fact, our condition covers all the examples in the literature, where this phenomena has been observed, and it does not only strengthen the result in a few cases but also allows a geometric interpretation. 

Our article begins with Section \ref{se.preliminiaries} where we fix most of our notations and introduce cylindrical L\'evy processes and the stochastic integrals. In Section \ref{se.Fubini} we present and establish the stochastic Fubini theorem for 
stochastic integrals with respect to cylindrical L\'evy processes  and deterministic integrands. In the following Section \ref{se.Cauchy} we apply the stochastic Fubini theorem to derive the existence of the weak solution of the stochastic Cauchy problem \eqref{SCP10}. In the final Section \ref{se.properties} we present some fundamental properties of the solution.

\section{Preliminaries}\label{se.preliminiaries}


Let $U$  and $V$ be  separable Hilbert spaces with norms $\|\cdot\|$ and orthonormal bases $(e_k)_{k\in \mathbb{N}}$ and $(h_k)_{k\in \mathbb{N}}$, respectively. We identify the dual of a Hilbert space by the space itself. The Borel $\sigma$-algebra of $U$ is denoted by $\Borel(U)$.
The space of Radon probability measures on $\Borel(U)$ is denoted by 
${\mathcal M}(U)$ and is equipped with the Prokhorov metric.
The space of all linear, bounded operators from $U$ to $V$ is 
denoted by $\mathcal{L}(U,V)$, equipped with the operator norm $\|\cdot \|_{\text{op}}$; its subset of Hilbert-Schmidt operators is denoted by
$\mathcal{L}_2(U,V)$, equipped with the norm $\|\cdot \|_{\text{HS}}$.
It follows from the standard characterisation of compact sets in Hilbert spaces, that a set $K \subseteq \mathcal{L}_2(U,V)$ is compact if and only if it is closed, bounded and satisfies
\begin{equation}\label{hscompact}
\lim \limits_{N \rightarrow \infty}\sup \limits_{\phi \in K}\sum_{k=N+1}^{\infty}\|\phi e_k\|^2=0.
\end{equation}
The space of all continuous functions from $[0,T]$ to $U$ is denoted by 
$C([0,T];U)$ and it is equipped with the supremum norm $\|\cdot                                                                                                                                                                                                                                                                                                                                                                                                                                                                                                                                                                                                                                                                                                                                                                                                                                                                                                                                                                                                                                                                                                                                                                                                                                                                                                                                                                                                                                                                                                                                                                                                                                                                                                                                                                                                                                                                                                                                                                                                                                                                                                                                                                                                                                                                                                                                                                                                                                                                                                                                                                                                            \|_{\infty}$. The space of all equivalence classes of measurable functions $ f\colon \Omega \rightarrow U$ on a probability space $(\Omega, \mathcal{F}, P)$
is denoted by $L^0_P(\Omega; U)$, and it is equipped with the topology of convergence in probability. The space $L^p_P(\Omega;U)$ contains all equivalence classes of measurable functions $f\colon\Omega\to U$ which are $p$-th integrable,
and it is equipped with the usual norm. 

For a subset $\Gamma$ of $U$, sets of the form 
\[ C(u_1, ... , u_n; B) :=\{ u \in U: (\langle u, u_1 \rangle, ... , \langle u, u_n \rangle) \in B\},\]
for $u_1, ... , u_n \in \Gamma$ and $B\in \Borel(\mathbb{R}^n)$ are called {\em cylindrical sets with respect to $\Gamma$}; the set of all these cylindrical sets is denoted by $\mathcal{Z}(U,\Gamma)$ and it is a $\sigma$-algebra if $\Gamma$ is finite and otherwise an algebra. A function $\mu \colon \mathcal{Z}(U,U) \to [0,\infty]$ is called a \emph{cylindrical measure}, if for each finite subset $\Gamma \subseteq U$ the restriction of $\mu$ on the $\sigma$-algebra $\mathcal{Z}(U, \Gamma)$ is a                                                                                                                                                                                                                                                                                                                                                                                                                                                                                                                                                                                                                                    measure. A cylindrical measure is called finite if $\mu (U) < \infty$ and a cylindrical probability measure if $\mu(U) =1$.
A \emph{cylindrical random variable} $Z$ in $U$ is a linear and continuous map
\[Z\colon U \rightarrow L_P^0(\Omega; \mathbb{R}).\]
Each cylindrical random variable $Z$ defines a cylindrical probability measure $\lambda$ by 
\begin{align*}
\lambda\colon  \mathcal{Z}(U,U) \to [0,1],\qquad
\lambda(C)=P\big( (Zu_1,\dots, Zu_n)\in B\big),
\end{align*}
for cylindrical sets $C=C(u_1, ... , u_n; B)$. The cylindrical probability measure $\lambda$ is called the {\em cylindrical distribution} of $Z$. The characteristic function of a cylindrical random variable $Z$ is defined by
\[\phi_{Z}\colon U \rightarrow \mathbb{C}, \qquad \phi_{Z}(u)=E[\exp (iZu)],\]
and it uniquely determines the cylindrical distribution of $Z$.

A family $(L(t):\, t\ge 0)$ of cylindrical random variables is called 
a {\em cylindrical process}. It is called a  \emph{cylindrical L\'evy process} if for all $u_1, ... , u_n \in U$ and $n\in \mathbb{N}$, the stochastic process
$((L(t)u_1, ... , L(t)u_n): t \ge 0)$
is a L\'evy process in $\mathbb{R}^n$. This concept is introduced in \cite{app} and it naturally extends the notion of a cylindrical Brownian motion.  The characteristic function of $L(t)$ for all $t \ge 0$ is given by
\[\phi_{L(t)} \colon U \rightarrow \mathbb{C}, \qquad \phi_{L(t)}(u)=\exp\big(t\Psi(u)\big),\]
where $\Psi \colon U \rightarrow \mathbb{C}$ is called the symbol of $L$, and is of the form
\[\Psi(u) = ia(u) - \frac{1}{2}\langle Qu, u\rangle +\int_U\left(e^{i\langle u, h \rangle}-1-i\langle u, h \rangle \1_{B_{\mathbb{R}}}(\langle u, h \rangle)\right)\mu (\udd h),\]
where $a \colon U \rightarrow \mathbb{R}$ is a continuous mapping with $a(0)=0$, $Q \colon U \rightarrow U$ is a positive, symmetric operator  and $\mu$ is a cylindrical measure on $\mathcal{Z}(U,U)$ satisfying
\[\int_U \left( \langle u, h \rangle^2 \wedge 1 \right) \mu(\udd h) < \infty \qquad \mathrm{for\;all\;}u \in U.\]
We call $(a, Q, \mu)$ the \emph{(cylindrical) characteristics of $L$};
see \cite{Riedle_infinitely}.


A function $g\colon [0,T] \to U$ is called {\em regulated} if $g$ has only discontinuities of the first kind. The space of all regulated functions 
is denoted by $R([0,T];U)$ and it is a Banach space when equipped 
with the supremum norm; see \cite[Ch.II.1.3]{bourbaki} for this and other properties 
we will use. A function $f\colon [0,T] \to \mathcal{L}(U,V)$ is called \emph{weakly in $R([0,T];U)$} if $f^*(\cdot)v$ is in $R([0,T];U)$ for each $v \in V$. For such a function one can introduce the stochastic integral
$\int_0^T \1_A(t)f^\ast(t)v\ud L(t)$ for all $v\in V$ and $A\in\Borel([0,T])$, which defines  
a cylindrical random variable
\[Z_A \colon V \to L_P^0(\Omega; \mathbb{R}), \qquad Z_Av= \int_0^T \1_A(t)f^*(t)v\ud L(t).\]
A function $f \colon [0,T] \to\mathcal{L}(U,V)$ is called \emph{stochastically integrable with respect to $L$} if $f$ is weakly in $R([0,T];U)$ and if for each $A \in \Borel([0,T])$ there exists a $V$-valued random variable $I_A$ such that
\begin{align*}
\langle I_A, v \rangle = Z_Av  \qquad \mathrm{for\;all\;} v \in V .
\end{align*}
 The stochastic integral $I_A$  is also denoted by $\int_Af(s)\ud L(s):=I_A$. 
From the very definition it follows that 
\begin{align} \label{eq.change-int-product}
\left\langle \int_A f(s)\ud L(s), v\right\rangle = \int_A f^*(s)v\ud L(s)
\qquad\text{for all }v\in V.
\end{align}
Necessary and sufficient conditions for the stochastic integrability of a function are derived in the work \cite{OU}. In particular, for Hilbert spaces $U$ and $V$ it states that
a function $f \colon [0,T] \to \mathcal{L}(U,V)$, which is weakly in $R([0,T];U)$, is stochastically integrable with respect to a cylindrical L\'evy process with characteristics $(a,Q,\mu)$  if and only if the following is satisfied:
\begin{enumerate}
\item[(1)] for every sequence $(v_n)_{n\in \mathbb{N}} \subseteq V$ converging weakly to $0$ and $A \in \Borel([0,T])$ we have
	\beq \lim_{n\to\infty}\int_Aa(f^*(s)v_n)\ud s=0.\label{sicond12}\eeq
\item[(2)]  \itemEq{\int_0^T \text{tr} \big[f(t)Qf^*(t)\big]\ud t <\infty;\label{sicond2}}
\item[(3)] \itemEq{\limsup \limits_{m\to \infty}\sup \limits_{n\ge m}\int_0^T\int_U\left(\sum_{k=m}^{n}\langle f(t) u, h_k\rangle^2\wedge 1\right)\mu(\udd u)\ud t=0.\label{sicond3}}
\end{enumerate}

	\section{Stochastic Fubini Theorem}\label{se.Fubini}

In this section, we prove a stochastic version of Fubini's theorem, which will play an essential role later. As the cylindrical L\'evy process $L$  does not enjoy a L\'evy-It{\^o} decomposition in $U$ we cannot exploit standard arguments. We will always denote by $(a,Q,\mu)$ the characteristics of $L$. Let $(S, \mathcal{S},\eta)$ be a finite measure space and $L^2_{\eta}(S;U)$ the Bochner space.
 \begin{thm}\label{SFT}
	 	Let $g\colon  S \times [0,T] \to U$ be a function satisfying the following assumptions:
	 	\begin{enumerate}[{\rm(a)}]
	 		\item $g$ is $\mathcal{S} \otimes \Borel([0,T])$ measurable;
	 		\item  the map  $t \mapsto g(s,t)$ is regulated for $\eta$-almost all $s\in S$;
	 		\item the map $t \mapsto g(\cdot,t)$ belongs to $R([0,T];L^2_\eta(S;U))$.
	 		\end{enumerate}
	 	 Then, $P$-almost surely, we have
\begin{align*}
	 		\int_S\int_0^Tg(s,t) \ud L(t)\,\eta (\udd s) = \int_0^T\int_Sg(s,t) \,\eta(\udd s)\ud  L(t),
\end{align*}
and all integrals are well defined; in particular, we have
\begin{enumerate}
\item[{\rm (1)}] 
 the map $t \mapsto \int_S g(s,t)\,\eta(\udd s)$ is in $R([0,T];U)$;
\item[{\rm (2)}]  
the process $\Big(\int_0^Tg(s,t) \ud L(t):\, s\in S\Big)$
defines a random variable in $L^2_\eta(S;\R)$. 
\end{enumerate}
	 \end{thm}

We divide the proof of the theorem in several lemmas.
The theory of  integration  developed in  \cite{OU} applies to deterministic integrands
$\Phi\colon [0,T]\to \mathcal{L}(U,V)$ which are regulated. In this case, the function $\Phi$ is
integrable if and only if it satisfies the Conditions \eqref{sicond12}, \eqref{sicond2} and \eqref{sicond3}. The following lemma shows that if $\Phi$
is Hilbert-Schmidt valued these conditions are already satisfied, i.e.\ $\Phi$ is stochastically integrable. This is in line with the general integration theory for random integrands developed in \cite{adam}, where the random integrands are assumed to have c{\'a}gl{\'a}d trajectories. 
\begin{lemma} \label{phistochint}  
Every regulated function 
$\Phi\colon [0,T]\to \mathcal{L}_2(U,V)$ is stochastically integrable with respect to $L$.
\end{lemma}
\begin{proof} 	
From the inequality $\|\Phi^*(t)v\|\le \|\Phi(t)\|_{\text{HS}}\|v\|$
for all $v\in V$ and a Cauchy argument, it follows that the operator $\Phi$ is weakly in $R([0,T];U)$.
We prove the stochastic integrability of $\Phi$ by verifying  Conditions  \eqref{sicond12},  \eqref{sicond2} and \eqref{sicond3}. To verify \eqref{sicond12}, let $v_n \to 0$ weakly in $V$. As the operator $\Phi^*(t)$ is compact for
each $t\in [0,T]$, it follows  $\Phi^*(t)v_n \to 0$ in the norm topology of $U$.
	 	Since $a$ is continuous and maps bounded sets to bounded sets by Lemma 3.2 in \cite{OU},
Lebesgue's theorem on dominated convergence implies
	 	\[\int_Aa(\Phi^*(t)v_n)\ud t \to 0 \qquad \text{as $n \to \infty$ for each $A\in \Borel([0,T])$.}
	\]
Since the mapping $t\mapsto \Phi(t)$ is regulated and thus bounded,
we obtain
	 	\begin{align*}	
	 	\int_0^T \text{tr} \big[\Phi(t)Q\Phi^*(t)\big]\ud t	
	 	& = \int_0^T \left\|\Phi(t)Q^{\frac{1}{2}}\right\|_{\text{HS}}^2\ud t
	 	< \infty,
	 	\end{align*}
which shows Condition  \eqref{sicond2}.
To prove Condition \eqref{sicond3}, note that
the monotone convergence theorem guarantees
	 	\begin{align}\label{eq.prove-C3}
\sup \limits_{n\geq m}\int_0^T\int_U\left(\sum_{k=m}^{n}\langle \Phi(t) u, h_k\rangle^2\wedge 1\right)\mu(\udd u)\ud t
=\int_0^T f_m(t) \ud t,
	 	\end{align}
where for each $m\in\N$ and $t\in [0,T]$ we define
	 	\begin{align*}
	 	f_m(t) &:=\sup \limits_{n\geq m}\int_U\left(\sum_{k=m}^{n}\langle \Phi(t) u, h_k\rangle^2\wedge 1\right)\mu(\udd u).
	 	\end{align*}
Let $\lambda$ denote the cylindrical distribution of $L(1)$. 
As $\Phi(t)$ is Hilbert-Schmidt for each fixed $t\in [0,T]$,
the image measure $\lambda\circ \Phi^{-1}(t)$ is a genuine infinitely divisible 
measure with  classical L\'evy measure $\mu\circ \Phi^{-1}(t)$. Consequently, we can apply
the monotone convergence theorem and Lebesgue's theorem on dominated convergence
to obtain for each $t\in [0,T]$ that
	 	\begin{align}
	 	 f_m(t)&= \sup \limits_{n\geq m}\int_V\left(\sum_{k=m}^{n}\langle v, h_k\rangle^2\wedge 1\right)(\mu \circ \Phi^{-1}(t))(\udd v)\notag \\
	 	& = \int_V\left(\sum_{k=m}^{\infty}\langle v, h_k\rangle^2\wedge 1\right)(\mu \circ \Phi^{-1}(t))(\udd v) \to 0 \qquad\text{as  } m \to \infty. \label{eq.limit=0}
	 	\end{align}
Since the set $K:= \overline{\{\Phi(t):t\in [0,T]\}}$ is a compact subset of $\mathcal{L}_2(U,V)$ by Problem 1 in  \cite[Ch.VII.6]{Dieudonne},
Proposition 5.3 in \cite{adam} implies that the set $\{\lambda \circ \phi^{-1}: \phi \in K\}$ is relatively compact in the space of probability measures on $\Borel(V)$. Since $\lambda \circ \phi^{-1}$  is
infinitely divisible with L\'evy measure
$\mu\circ \phi^{-1}$,  Theorem VI.5.3
in \cite{partha} implies 
\begin{align*}
	\sup \limits_{\phi\in K} \int_{\|v\| \leq 1}\|v\|^2  (\mu \circ \phi^{-1})(\udd v)<\infty
\qquad\text{and}\qquad	
 \sup \limits_{\phi \in K}  (\mu \circ \phi^{-1})(\{v:\|v\| >1\})<\infty.
\end{align*}
Consequently, we obtain
	 	\begin{align}
	 	\sup_{m \in \N}\sup \limits_{t\in [0,T]} 	f_m(t)
	 	& \le 	\sup \limits_{\phi\in K} \int_{\|v\| \leq 1}\|v\|^2  (\mu \circ \phi)(\udd v) + \sup \limits_{\phi \in K} \int_{\|v\| >1} (\mu \circ \phi)(\udd v) < \infty. \label{eq.bounded}
	 	\end{align}
The limit \eqref{eq.limit=0} and the inequality \eqref{eq.bounded}
enable us to apply Lebesgue's theorem in \eqref{eq.prove-C3},
which proves Condition \eqref{sicond3}.
	 \end{proof}
	 
For some $u\in U$ and $v\in V$, we define the operator $u\otimes v\colon  U \to V$ by $(u \otimes v)(w):=\langle u, w\rangle v$. 

	 \begin{lemma}\label{discretize2}
	 	If $\Phi\colon [0,T]\to \mathcal{L}_2(U,V)$ is a regulated function, then  $\sum \limits_{j=1}^{m}e_j \otimes \Phi(\cdot)e_j$ converges to  $\Phi$ in $R\big([0,T],\mathcal{L}_2(U,V)\big)$ as $m \to \infty$.
	 \end{lemma}
	 \begin{proof}
	 	 By Problem 1 in  \cite[Ch.VII.6]{Dieudonne},  the set $K:= \overline{\{\Phi(t):t\in [0,T]\}}$ is  compact in $\mathcal{L}_2(U,V)$. By applying  \eqref{hscompact} we conclude 
	 	 \begin{align*}
	 	 \sup \limits_{t\in[0,T]}\left \|\Phi(t)-\sum_{j=1}^{m}e_j \otimes \Phi(t)e_j \right \|^2_{\text{HS}}
	 	 & =\sup\limits_{t\in[0,T]}\sum_{i=1}^{\infty}\left \|\Phi(t)e_i-\sum_{j=1}^{m}\langle e_j,e_i\rangle \Phi(t)e_j \right \|^2\\
	 	 & =\sup\limits_{t\in[0,T]}\sum_{i=m+1}^{\infty}\|\Phi(t)e_i\|^2\\
	 	 & \leq \sup\limits_{\phi \in K}\sum_{i=m+1}^{\infty}\|\phi e_i\|^2 \to 0 \text{ as } m \to \infty,                                                                                                                                                                                                                                                                                                                                                                                                                                                                                                         
	 	 \end{align*}
which completes the proof.
	 \end{proof}
	 \begin{lemma} \label{discretize}
For each regulated function $\Phi\colon [0,T] \to \mathcal{L}_2(U,V)$ there exists a sequence of partitions $\{(t_k^n)_{k=0}^{N_n}: n \in \mathbb{N}\}$  of $[0,T]$ with $\max_{0\leq k \leq N_n-1}|t^n_{k+1}-t^n_{k}|\to 0$  as $ n\to \infty$ such that the functions
	 	\begin{equation} \label{phi_nm}
	 	\Phi_{m,n}(t):= 
\begin{cases}\displaystyle
\sum_{j=1}^{m}e_j \otimes \Phi\left(\tfrac{t^{n}_{k}+t^{n}_{k+1}}{2}\right)e_j, & \text{if }  t \in (t_k^n, t_{k+1}^n),\, k=0, \ldots, N_n-1,\\
\displaystyle
\sum_{j=1}^{m}e_j \otimes \Phi(t^{n}_k)e_j, & \text{if }   t=t_k^n,\, k=0, \ldots, N_n,\\
\end{cases}
	\end{equation}
	 	satisfy \beq \label{Phinmconv} \lim\limits_{m,n \to \infty}\sup \limits_{t\in[0,T]}\|\Phi_{m,n}(t)- \Phi(t)\|_{\text{HS}}=0,\eeq and
	 	\begin{equation}\label{convinprob}
	 	\lim\limits_{m,n \to \infty}\int_0^T\Phi_{m,n}(t)\ud L(t) = \int_0^T\Phi(t)\ud L(t)\qquad \text{in probability.}\end{equation}
	 \end{lemma}
	 \begin{proof}
 Using \cite[7.6.1]{Dieudonne}, we can construct a sequence $\{(t_k^n)_{k=0}^{N_n}: n \in \mathbb{N}\}$ of  partitions of $[0,T]$ such that $\max \limits_{0\leq k \leq N_n-1}|t^n_{k+1}-t^n_{k}|\to 0$ and that the functions
\begin{align*}
\Phi_n(t):= \left\{\begin{array}{ll}\Phi\left(\frac{t^{n}_{k}+t^{n}_{k+1}}{2}\right), & \text{if  } t \in (t_k^n, t_{k+1}^n),\, k=0, \ldots, N_n-1,\\
\Phi(t^n_k), & \text{if }t = t^n_k,\, k=0, \ldots, N_n, \end{array}
\right.
\end{align*}
     satisfy
	 	\begin{align*}
	 	\sup \limits_{t\in [0,T]}\|\Phi(t)-\Phi_n(t)\|_{\text{HS}}  \to 0\quad \text{ as } n \to \infty.
	 	\end{align*}
	 	Let $\epsilon >0$ be given. Then there exists $N \in \mathbb{N}$ such that for all $n \geq N$, we have
	 	\begin{equation}\label{phi_n}
	 	\sup \limits_{t\in [0,T]}\|\Phi(t)-\Phi_n(t)\|_{\text{HS}} \le \frac{\epsilon}{2}.
	 	\end{equation}
	 		By Lemma \ref{discretize2}, there exists $M >0$, such that for all $m \geq M$, we have
	 	\begin{equation} \label{phi_tm}
	 	\sup\limits_{t\in[0,T]} \left\|\Phi(t)-\sum_{j=1}^{m}e_j \otimes \Phi(t)e_j \right\|_{\text{HS}} \le \frac{\epsilon}{2}.
	 	\end{equation}
	 	Using \eqref{phi_n} and \eqref{phi_tm}  we have for all $n\geq N$ and $m \geq M$,
	 	\begin{align*}
	 	\sup \limits_{t\in [0,T]} &\|\Phi(t)-\Phi_{m,n}(t)\|_{\text{HS}}\\
	 	& \leq \sup \limits_{t\in [0,T]}\|\Phi(t)-\Phi_{n}(t)\|_{\text{HS}}+ \sup \limits_{t\in [0,T]}\|\Phi_{n}(t)-\Phi_{m,n}(t)\|_{\text{HS}}\\
	 	& =  \sup \limits_{t\in [0,T]}\|\Phi(t)-\Phi_{n}(t)\|_{\text{HS}}+ \sup \limits_{t\in [0,T]}\left(\sum_{k=0}^{N_{n}}\1_{\{t^{n}_{k}\}}(t)\left\|\Phi(t_k^{n})-\sum_{j=1}^{m}e_j \otimes \Phi(t^{n}_k)e_j \right\|_{\text{HS}}\right.\\
& \qquad \qquad \left.+\sum_{k=0}^{N_{n}-1}\1_{(t^{n}_{k},t^{n}_{k+1})}(t)\left\|\Phi\left(\tfrac{t^{n}_{k}+t^{n}_{k+1}}{2}\right)-\sum_{j=1}^{m}e_j \otimes \Phi\left(\tfrac{t^{n}_{k}+t^{n}_{k+1}}{2}\right)e_j \right \|_{\text{HS}}\right)\\
	 	& \leq \frac{\epsilon}{2}+\frac{\epsilon}{2} = \epsilon,
	 	\end{align*} 
which proves  \eqref{Phinmconv}.
Let $P_{m,n}$ denote the probability distribution of $\int_0^T \Phi_{m,n}(t) \ud  L(t)$.  For establishing \eqref{convinprob}, it is sufficient by   \cite[Lemma 2.4]{jaku1} to show: 
	 	\begin{enumerate}[(i)]
	 		\item \label{convweak} $\left \langle \int_0^T \Phi_{m,n}(t) \ud  L(t) - \int_0^T \Phi(t) \ud L(t), v \right \rangle \rightarrow 0\quad \text{in probability for all}\;v \in V;$
	 		\item \label{tightness}
$\{P_{m,n}: m,n \in \mathbb{N} \}$ is relatively compact in ${\mathcal M}(V)$.
	 	\end{enumerate}
As $\Phi^\ast_{m,n}(\cdot)v $ converges uniformly 
to $\Phi^\ast(\cdot)v$ for each $v\in V$ due to \eqref{Phinmconv}, Lemma 5.2 in \cite{OU} implies
\[\left \langle \int_0^T \Phi_{m,n}(t) \ud  L(t) - \int_0^T \Phi(t) \ud L(t), v \right \rangle 
=\int_0^T \big(\Phi_{m,n}^\ast(t)-\Phi^\ast(t)\big)v\ud L(t)
\rightarrow 0 \]
in probability which establishes \eqref{convweak}.  To prove \eqref{tightness}, we define the set
	 	\[K_1 := \left\{\sum_{j=1}^{m}e_j \otimes \phi e_j: m \in \mathbb{N} \cup \{\infty\},  \phi \in K\right\},\]
	 	where $K:= \overline{\{\Phi(t):t\in [0,T]\}}$. 
Since $K$ is a compact subset of $\mathcal{L}_2(U,V)$, it follows that $K_1$ is closed and bounded. By applying \eqref{hscompact} we obtain
	 	\begin{align*}
	 	\lim \limits_{N \rightarrow \infty}\sup \limits_{\psi \in K_1}\sum_{k=N+1}^{\infty}  \|\psi e_k\|^2& =\lim \limits_{N \rightarrow \infty}\sup \limits_{\phi \in K} \sup \limits_{m \in \mathbb{N}\cup \{\infty\}}\sum_{k=N+1}^{\infty} \left\|\sum_{j=1}^{m}\langle e_j, e_k \rangle \phi e_j\right\|^2\\
	 	&= \lim \limits_{N \rightarrow \infty}\sup \limits_{\phi \in K} \sum_{k=N+1}^{\infty} \|\phi e_k\|^2=0,
	 	\end{align*}
which shows that $K_1$  is a compact subset of $\mathcal{L}_2(U,V)$. Proposition 5.3 in  \cite{adam} guarantees that the set $\{\lambda \circ \psi^{-1}: \psi \in K_1\}$,   is relatively compact in the space of probability measures on $\Borel(V)$, where $\lambda$ is the cylindrical distribution of $L(1)$.
Since
	 	\[P_{m,n} = (\lambda \circ (\psi^n_{m,1})^{-1})^{*(t_1^n-t_0^n)}* \cdots *(\lambda \circ (\psi^n_{m,N_n-1})^{-1})^{*(t_{N_n}^n-t_{N_n-1}^n)},\]
	 	where $\psi^n_{m,k}:= \sum_{j=1}^{m}e_j \otimes \Phi\left(\frac{t^{n}_{k}+t^{n}_{k+1}}{2}\right)e_j$ and $\psi^n_{m,k}$ is 
	 	in the compact set $ K_1$ for each $k\in \{0,\dots, N_n\}$,   Lemma~5.4 in \cite{adam} implies  \eqref{tightness}.
	 \end{proof}

	\begin{lemma} \label{isometric}
		The mapping 
\begin{align*}		
J\colon  R\big([0,T];L^2_\eta(S;U)\big) \to  R\big([0,T];\mathcal{L}_2(U,L^2_\eta(S;\R))\big),
\qquad		J(f)(t)u = \langle u, f(t)(\cdot)\rangle,
\end{align*}
is a well defined isometric isomorphism.
	\end{lemma}
\begin{proof}			
For each $t\in[0,T]$ and $f \in R\big([0,T];L^2_\eta(S;U)\big)$, the map $J(f)(t)$ defines a linear and continuous operator from $U$ to $L^2_\eta(S;\R)$ and satisfies
				\begin{align}\label{eq.J-isometric}
\|J(f)(t)\|_{\rm{HS}}^2
= \sum_{j=1}^{\infty}\int_S\langle e_j, f(t)(s)\rangle^2 \,\eta(\udd s)
= \|f(t)\|^2_{L^2_\eta(S;U)}.
       			\end{align}
As $t\mapsto f(t)$ is regulated, the isometry \eqref{eq.J-isometric} shows by a Cauchy argument that   $t\mapsto J(f)(t)$ is regulated. Consequently, $J$ is a well defined linear isometry and it is left 
to show that $J$ is surjective.

For this purpose, let $ \Phi $ be in $R\big([0,T];\mathcal{L}_2(U,L^2_\eta(S;\R))\big)$. We define
\begin{align*}
f\colon [0,T] \rightarrow L^2_\eta(S;U),
\qquad
	f(t)(\cdot):= \sum_{j=1}^{\infty}\left(\Phi(t)e_j\right)(\cdot)e_j,
\end{align*}
where the series converges in $L^2_\eta(S;U)$. As $\|f(t)\|_{L^2_\eta(S;U)}=\|\Phi(t)\|_{\mathcal{L}_2(U,L^2_\eta(S;\R))}$, the function $t\mapsto f(t)$ is regulated and satisfies
\begin{align*}
(J(f)(t))(u)  =\sum_{j=1}^{\infty}\Phi(t)e_j(\cdot)\langle u, e_j\rangle
 =\sum_{j=1}^{\infty}\Phi(t)\big(\langle u, e_j\rangle e_j\big)(\cdot)
 =\Phi(t)\left( u\right)(\cdot),
		\end{align*}
which completes the proof.
		\end{proof}
\begin{lemma}\label{le.g_nm}
Let $g\colon S\times [0,T]\to U$ be a function 
such that 
the map  $t \rightarrow g(s,t)$ is regulated for $\eta$-almost all $s\in S$,
and $\{(t_k^n)_{k=0}^{N_n}: n \in \mathbb{N}\}$ be a sequence of
 partitions of $[0,T]$ with $\max_{0\leq k \leq N_n-1}|t^n_{k+1}-t^n_{k}|\rightarrow 0$.
Then  the functions $g_{m,n}\colon S\times [0,T] \rightarrow U$ defined by
\begin{align}\label{eq.g_nm} 
 g_{m,n}(s,t):=
\begin{cases}\displaystyle
 \sum_{j=1}^{m} \left\langle e_j,g\left(s,\tfrac{t^n_k+t^n_{k+1}}{2}\right)\right\rangle e_j & \text{if }t\in (t_{k}^n,t_{k+1}^n), k=0, \ldots, N_n-1, \\
 \displaystyle
\sum_{j=1}^{m} \langle e_j,g(s,t^n_k)\rangle e_j, & \text{if }t=t_{k}^n, k=0, \ldots, N_n, \end{cases}
\end{align}
satisfy for $\eta$-almost all $s\in S$ that 
\begin{align*}
\|g_{m,n}(s,t)-g(s,t)\| \to 0 \quad \text{for Lebesgue-almost all } t\in [0,T]. 
\end{align*}
\end{lemma}

\begin{proof}
For each $n \in \mathbb{N}$, define $g_n\colon S\times [0,T] \rightarrow U$ by
	\[g_n(s,t):= \sum_{k=0}^{N_n-1}\1_{(t_{k}^n,t_{k+1}^n)}(t)g\left(s,\tfrac{t^n_k+t^n_{k+1}}{2}\right) +\sum_{k=0}^{N_n-1}\1_{\{t_{k}^n\}}(t)g(s,t_k^n) .\]
Let $s\in S$ be such that $g(s,\cdot)$ is regulated. Then the set 
$A_s\subseteq [0,T]$ of discontinuities of $g(s,\cdot)$ has Lebesgue measure $0$ 
and for each $t\in A_s^c$ it follows that
\begin{align}\label{gnlim}
\lim_{n\to\infty}\|g_n(s,t)-g(s,t)\|=0.
\end{align}
The set $\overline{\{g(s,t):t\in [0,T]\}}$ is compact in $U$ by Problem 1 in \cite[VII.6]{Dieudonne}. The compactness criterion in Hilbert spaces implies
	\begin{align*}
	\sup \limits_{t\in [0,T]}\left \|\sum_{j=1}^m\langle g(s,t),e_j\rangle e_j-g(s,t)\right \|^2
&= \sup \limits_{t\in [0,T]}\sum_{j=m+1}^{\infty}\langle g(s,t),e_j\rangle^2\\
&\rightarrow 0 \quad\text{  as } m \rightarrow \infty. \numberthis \label{supgn}
	\end{align*}
By using \eqref{gnlim} and \eqref{supgn} we obtain for each $t\in A_s^c$ that
	\begin{align*}
&	\|g_{m,n}(s,t)-g(s,t)\| \\
&\qquad\qquad \leq \|g_{m,n}(s,t)-g_n(s,t)\|+ \|g_{n}(s,t)-g(s,t)\|\\
&\qquad\qquad = \sum_{k=0}^{N_n-1}\1_{(t_{k}^n,t_{k+1}^n)}(t)\left\|\sum_{j=1}^{m} \left\langle g\left(s,\tfrac{t^n_k+t^n_{k+1}}{2}\right),e_j\right\rangle e_j-g\left(s,\tfrac{t^n_k+t^n_{k+1}}{2}\right)\right\|\\
& \qquad\qquad\qquad +\sum_{k=0}^{N_n}\1_{\{t_{k}^n\}}(t)\left\|\sum_{j=1}^{m} \langle g(s,t_k^n),e_j\rangle e_j-g(s,t_k^n)\right\|+ \|g_{n}(s,t)-g(s,t)\|\\
&\qquad\qquad \leq \sup \limits_{r\in [0,T]}\left\|\sum_{j=1}^m\langle g(s,r),e_j\rangle e_j-g(s,r)\right\|+\|g_{n}(s,t)-g(s,t)\|\\
&\qquad\qquad \rightarrow 0 \qquad\text{as  } m,n \to\infty,
	\end{align*}
which completes the proof.
\end{proof}

In the following we extend our definition of the space $L^0_P(\Omega;U)$ 
to a finite measure space $(A,\mathcal{A},\sigma)$ and a complete metric space $(E,d)$. In this case, $L_\sigma^0(A;E)$ denotes the space of the equivalence classes of all separably-valued, measurable functions from $A$ to $E$. As before, the space is an $F$-space equipped with the metric
\begin{align}\label{KyFanMetric}
 \rho (f,g):= \int_A\big( d(f(x),g(x))\wedge 1\big)\,\sigma(\udd x).
\end{align}	
Instead of separably-valued, one can equivalently require strong measurability of the functions. 
 \begin{lemma}\label{le.L0_isomorphism} Let $(A_1, \mathcal{A}_1,\sigma_1)$ and $(A_2, \mathcal{A}_2,\sigma_2)$ be two finite measure spaces and $V$ be a separable Hilbert space. Then
	 \[L^0_{\sigma_1}\big(A_1;L^0_{\sigma_2}(A_2;V)\big)\cong L^0_{\sigma_1 \otimes \sigma_2}\big(A_1\times A_2;V\big).\]
In particular, the isomorphism is given such that for each $\mathcal{A}_1$-measurable function $F \colon A_1 \to L^0_{\sigma_2}(A_2;V)$, there corresponds an $\mathcal{A}_1 \otimes \mathcal{A}_2$-measurable function $f\colon A_1 \times A_2 \to V$ such that for $\sigma_1$-almost all $x \in A_1$, we have $F(x) = f(x,\cdot)$ in $L^0_{\sigma_2}(A_2;V)$, and conversely. 	 
\end{lemma}
			 \begin{proof} The lemma can be proved similarly as Lemma III.11.16 in  \cite{Dunford_Schwartz} by replacing $L^p$-norms for $p\ge 1$ by the corresponding  metrics as defined in \eqref{KyFanMetric}. 
			 	\end{proof}

	 \begin{proof}[Proof of Theorem \ref{SFT}]
	Lemma \ref{isometric} guarantees that the mapping 
	\begin{align*} 
	 \Phi\colon [0,T] \rightarrow \mathcal{L}_2(U,L^2_\eta (S;\R)),\qquad
	\Phi(t)u:= \langle u, g(\cdot,t)\rangle,
	\end{align*}
is well defined and regulated. Let $\Phi_{m,n}$ denote the functions defined in 
\eqref{phi_nm} for $V=L^2_\eta(S;\R)$. Lemma~\ref{discretize} together with Lemma \ref{le.L0_isomorphism} 
imply, upon passing to a subsequence, that 
for $(\eta\otimes P)$-almost all $(s,\omega) \in S\times \Omega$ we have
\begin{align}\label{fb1}
	\left(\left(\int_0^T\Phi(t)\ud L(t)\right)(\omega)\right)(s) &=\lim \limits_{m,n\rightarrow \infty}\left(\left(\int_0^T\Phi_{m,n}(t)\ud L(t)\right)(\omega)\right)(s).
\end{align}
For each $h \in L^2_\eta(S;\R)$, we obtain by \eqref{eq.change-int-product} that
	\begin{align*}
&\left \langle \int_0^T \Phi_{m,n}(t) \ud  L(t), h \right \rangle_{L^2_\eta(S;\R)}\\
&\qquad\qquad = \int_0^T \Phi^*_{m,n}(t)h \ud  L(t)\\
&\qquad\qquad = \sum_{k=0}^{N_n-1}\big(L(t_{k+1})-L(t_{k})\big)\left(\sum_{j=1}^{m}\left\langle e_j, \Phi^*\left(\tfrac{t^n_k+t^n_{k+1}}{2}\right)h\right\rangle e_j\right)\\
&\qquad\qquad = \sum_{k=0}^{N_n-1}\sum_{j=1}^{m}\left\langle \Phi\left(\tfrac{t^n_k+t^n_{k+1}}{2}\right)e_j,h\right\rangle_{L^2_\eta(S;\R)}\big(L(t_{k+1})-L(t_{k})\big)(e_j)\\
&\qquad\qquad =\left \langle \sum_{k=0}^{N_n-1}\sum_{j=1}^{m} \Phi\left(\tfrac{t^n_k+t^n_{k+1}}{2}\right)e_j\big(L(t_{k+1})-L(t_{k})\big)(e_j),h\right\rangle_{L^2_\eta(S;\R)}.
	\end{align*}
Therefore, for $\eta$-almost all $s\in S$, we have
		\begin{align}\label{eq.int-Phi-g}
\left(\int_0^T \Phi_{m,n}(t) \ud  L(t)\right)(s)& = \left(\sum_{k=0}^{N_n-1}\sum_{j=1}^{m} \Phi\left(\tfrac{t^n_k+t^n_{k+1}}{2}\right)e_j\big(L(t_{k+1})-L(t_{k})\big)(e_j)\right)(s)\notag \\
		& = \sum_{k=0}^{N_n-1}\sum_{j=1}^{m} \left(\Phi\left(\tfrac{t^n_k+t^n_{k+1}}{2}\right)e_j\right)(s)\big(L(t_{k+1})-L(t_{k})\big)(e_j)\notag \\
		&= \sum_{k=0}^{N_n-1}\big(L(t_{k+1})-L(t_{k})\big)\left(\sum_{j=1}^{m} \left(\Phi\left(\tfrac{t^n_k+t^n_{k+1}}{2}\right)e_j\right)(s)e_j\right)\notag \\
		&= \int_0^Tg_{m,n}(s,t)\ud L(t),
		\end{align}
where $g_{m,n}$ denotes the functions defined in \eqref{eq.g_nm}.
By Lemma 5.4 in \cite{OU}, we have for each $ \alpha \in \mathbb{R}$ that
		\begin{align}\label{eq.char-int-g}
&E\left[ \exp\left(i \alpha \int_0^T\big(g_{m,n}(s,t)-g(s,t)\big)\ud L(t) \right)\right] \notag \\
&\qquad\qquad = \exp \left( \int_0^T\Psi\big(\alpha\left(g_{m,n}(s,t)-g(s,t)\right)\big)\ud t\right),
		\end{align}
where $\Psi$ denotes the L{\`e}vy symbol of $L$. Note that 
\begin{align*}
\sup_{t\in [0,T]}\|g_{m,n}(s,t)-g(s,t)\|^2 &  \leq  4\sup \limits_{t \in [0,T]}\|g(s,t)\|^2 < \infty. 	
\end{align*}
Since $\Psi$ is continuous and maps bounded sets to bounded sets according to Lemma~3.2 in \cite{OU},  it follows by Lebesgue's theorem on dominated convergence and Lemma~\ref{le.g_nm} that 
\begin{align*} 
\lim \limits_{m,n \rightarrow \infty} \int_0^T\Psi\left(\alpha(g_{m,n}(s,t)-g(s,t))\right)\ud t= 0.
\end{align*}
Consequently, we deduce from \eqref{eq.char-int-g} that for $\eta$-almost all $s\in S$,
	\begin{equation}
	\lim \limits_{m,n \rightarrow \infty}\int_0^Tg_{m,n}(s,t)\ud L(t)=\int_0^Tg(s,t)\ud L(t) \qquad \text{in probability}.\label{fb2} 
	\end{equation}                                                                                                                  Comparing limits in \eqref{fb1} and \eqref{fb2} by means of \eqref{eq.int-Phi-g}, we obtain that for $\eta$-almost all $s\in S$, we have
	\begin{equation} \label{aasaaw}
		\left(\int_0^Tg(s,t)\ud L(t)\right)(\omega) =\left(\left(\int_0^T\Phi(t)\ud L(t)\right)(\omega)\right)(s) \quad\text{for $P$-almost all $\omega\in \Omega$}.
	\end{equation}
By \eqref{fb2} and Lemma \ref{le.L0_isomorphism}, the left hand side in 
\eqref{aasaaw} is $\mathcal{S}\otimes \mathcal{F}$ measurable, as well as 
the right hand side due to \eqref{fb1}.
A further application of Fubini's theorem implies for $P$-almost all $\omega \in \Omega$ that
	\begin{equation*}
	\left(\int_0^Tg(s,t)\ud L(t)\right)(\omega) =\left(\left(\int_0^T\Phi(t)\ud L(t)\right)(\omega)\right)(s) \quad\text{for $\eta$-almost all $s\in S$.}
	\end{equation*}
By integrating both sides and denoting by $1$ the function in 
$L^2_\eta(S;\R)$ which constantly equals one,  we obtain by \eqref{eq.change-int-product} that
	\begin{align*}
	\int_S \left(\int_0^Tg(s,t)\ud L(t)\right)(\omega)\, \eta(\udd s) &= \int_S \left(\left(\int_0^T\Phi(t) \ud L(t)\right)(\omega)\right)(s) \,\eta(\udd s)\\
	&= \left \langle \left(\int_0^T\Phi(t)\ud L(t)\right)(\omega), 1\right \rangle_{L^2_\eta(S;\R)}\\
	&= \left(\int_0^T\Phi^*(t)1 \ud L(t)\right)(\omega)\\
	&=\left(\int_0^T\int_Sg(s,t)\eta(\udd s)\ud L(t)\right)(\omega),
	\end{align*} 	
which completes the proof.
	  \end{proof}


\section{Cauchy Problem}\label{se.Cauchy}
 We consider the following stochastic Cauchy problem driven by a cylindrical L\'evy process  $L$ in a separable Hilbert space $U$:
\begin{equation}\label{SCP}
\begin{split}
\ud Y(t)&= AY(t)\ud t+B\ud L(t)  \qquad  \mathrm{for\; all}\;  t \in [0,T], \\
Y(0)& =  y_0,
\end{split}
\end{equation}
where $A$ is a generator of a strongly continuous semigroup $(T(t))_{t\geq 0}$ on a separable Hilbert space $V$,  $B\colon U \rightarrow V$  is a linear and continuous operator and the initial condition $y_0$ is in $V$.

In the case of $L$ beeing a cylindrical Brownian motion, the concept of weak solution is defined in \cite{daprato_zab} and the existence and uniqueness of weak solution is established.  Their definition requires  weak solutions to have almost surely Bochner integrable paths. In case of Banach spaces, a similar definition is used in \cite{neerven}. However, as it is known that the solution of \eqref{SCP} may exhibit highly irregular paths, the requirement of Bochner integrable paths is too restrictive in our situation. A weaker condition requires only that the paths  $t \mapsto \langle Y(t), A^*v\rangle$ are integrable for $v \in \mathcal{D}(A^*)$; see \cite{brzezniak_neerven},  \cite{peszat_zabczyk} and  \cite{neerven_weis}. We will impose a slightly stronger condition but which is still weaker than Bochner integrability of the paths.

\begin{defn}
A V-valued stochastic process $(Y(t):t\geq 0)$ is called {\em weakly Bochner regular} if
$t\mapsto \scapro{Y(t)}{g(t)}$ is integrable on $[0,T]$ for each $g\in C([0,T];V)$
and for every sequence $(g_n)_{n\in \mathbb{N}} \subseteq C([0,T];V)$ with $\|g_n \|_{\infty} \rightarrow 0$ as $n \to \infty$, we have 
	\[ \int_0^T \langle Y(s),g_{n}(s)\rangle \ud s \rightarrow 0 \qquad \text{in probability as }k \to\infty.\]
\end{defn}
If the stochastic process $Y$ has Bochner integrable paths on $[0,T]$, then $Y$ is also weakly Bochner regular as shown by a simple estimate.

\begin{defn}\label{de.weak-solution}
A $V$-valued, progressively measurable stochastic process $(Y(t):\,t \in [0,T])$ is called a \emph{weak solution} of  the stochastic Cauchy problem \eqref{SCP} if $Y$ is weakly Bochner regular and satisfies for every $v \in \mathcal{D}(A^*)$ and $t\in [0,T]$, $P$-almost surely, that  
\begin{align}\label{eq.def-weak-eq}
\langle Y(t), v\rangle = \langle y_0,v\rangle +\int_0^t \langle Y(s),A^*v\rangle \ud s+L(t)(B^*v).
\end{align}
\end{defn}

\begin{thm} \label{existence}
If the mapping $s \mapsto T(s)B$ is stochastically integrable on $[0,T]$ with respect to $L$, then
	\[Y(t)=T(t)y_0+\int_0^tT(t-s)B\ud L(s), \quad t\in [0,T],\] 
is a  weak solution of the stochastic Cauchy problem \eqref{SCP}.
\end{thm}
\begin{example}\label{ex.series-existence}
In this and the next example we set $V=U$, $B = \text{Id}$ and 
assume that there exist $\lambda_k\ge 0$ with $\lambda_k  \rightarrow \infty$ as $k\rightarrow \infty$ such that 
\begin{align}\label{eq.semigroup-spectral}
T^*(t)e_k = e^{-\lambda_kt}e_k \qquad \text{for all } t\in[0,T], \, k\in \mathbb{N}.
\end{align}
In the literature, e.g.\ \cite{Brzetal}, \cite{LiuZhai}, \cite{LiuZhai16}
and \cite{PeszatZab12}, often  cylindrical L\'evy processes of the following form  are considered: 
\begin{align}\label{example.series}
L(t)u:=\sum \limits_{k=1}^{\infty}\langle e_k, u\rangle \sigma_k \ell_k(t)\qquad \text{for all } t\in[0,T], \, u\in U,
\end{align}
	where $(\ell_k)_{k\in \N}$ is a sequence of independent, symmetric, real valued L\'evy  processes with characteristics $(0,0,\mu_k)$ and 
$(\sigma_k)_{k\in\N}$ is a real valued sequence such 
that the series in \eqref{example.series} converges in 
$L_P^0(\Omega;\R)$.
By using \eqref{sicond3} we obtain that $T(\cdot)$ is stochastically integrable with respect to  $L$ if and only if
	\begin{align}
	\sum_{k=1}^{\infty}\int_0^T\int_\mathbb{R}\left( e^{-2\lambda_k s}|\sigma_k\beta|^2 \wedge 1\right) \mu_k (\udd \beta)\ud t < \infty;
	\end{align}
see Corollary 6.3 in \cite{OU}.
For example, if $(\ell_k)_{k\in \N}$ is a family of independent, identically distributed, standardised, symmetric $\alpha$-stable processes with $\alpha \in (0,2)$, one easily computes that $T(\cdot)$ is stochastically integrable w.r.t.\ $L$ if and only if 
	\begin{align}\label{sicond.series-stable}
\sum_{k=1}^{\infty}\frac{|\sigma_k|^{\alpha}}{\lambda_k} < \infty.	
\end{align}
This result on the existence of a weak solution of the stochastic 
Cauchy problem \eqref{SCP} coincides with the result in \cite{priola_zabczyk}.
\end{example}
\begin{example}\label{ex.canonical-stable-existence}
We assume the same setting as in Example \ref{ex.series-existence} but 
model $L$ as the canonical $\alpha$-stable cylindrical L\'evy process for $\alpha \in (0,2)$, i.e.\ the characteristic function of $L(t)$ is of the form
\[\phi_{L(t)}:U \rightarrow \mathbb{C}, \qquad \phi_{L(t)}(u)=\exp \left(-t\|u\|^{\alpha}\right).\]
Obviously, each finite dimensional projection $\big( (L(t)u_1,\dots
, L(t)u_n):\,t\ge 0\big)$ for $u_1,\dots, u_n\in U$ is an $\alpha$-stable  L\'evy process
in $\R^n$. Using this fact, it is shown in Theorem 4.1 in  \cite{Riedle_alpha_stable} that a semigroup  $(T(t))_{t\geq 0}$ satisfying 
the spectral decomposition \eqref{eq.semigroup-spectral}  is stochastically integrable with respect to $L$ if and only if 
\begin{align}\label{eq.canonical-integrable}
\int_0^T\|T(s)\|^{\alpha}_{\text{HS}}< \infty.
\end{align}
In the work \cite{brzezniak_zabczyk}, the authors consider the stochastic 
Cauchy problem in Banach spaces driven by a subordinated cylindrical Brownian motion, a slightly more general noise than the canonical $\alpha$-stable cylindrical L\'evy process. As the approach in \cite{brzezniak_zabczyk} relies on embedding the underlying space $U$ in a larger space, the derived conditions are less explicit than 
\eqref{eq.canonical-integrable} and only sufficient. 
\end{example}

\begin{proof}[Proof of Theorem \ref{existence}]
We can assume $y_0 =0$ due to linearity. Lemma 6.2 in \cite{OU} guarantees that the map  $r \mapsto T(s-r)B$ is stochastically integrable on $[0,s]$ for each $s\in (0,T]$. Thus, we can define
	\[Y(s):=\int_0^s T(s-r)B \ud L(r)\qquad\text{for all } s\in [0,T].\]
We first show that $Y$ is weakly Bochner regular. Let $g$ be in $C([0,T];V)$ and define
\begin{align}\label{f_alpha}
f\colon [0,T]\times [0,T]\to U, 
\qquad f(s,r)=\1_{[0,s]}(r)B^\ast T^\ast(s-r)g(s).
\end{align}
By using \eqref{eq.change-int-product} we conclude for all $s\in [0,T]$ that
\begin{align}\label{eq.convolution-g}
\scapro{Y(s)}{g(s)}=\int_0^s B^\ast T^\ast(s-r)g(s)\ud L(r)
=\int_0^T f(s,r)\ud L(r).
\end{align}
For fixed $s\in [0,T]$ the map $r\mapsto f(s,r)$ is regulated. Moreover, by defining $m:=\sup_{s\in [0,T]}\|B^\ast T^\ast (s)\|_{\text{op}}^2$,  we obtain for $\epsilon>0$ and $r \in [0, T-\epsilon]$ that
\begin{align*}
&\|f(\cdot,r+\epsilon)-f(\cdot,r)\|_{L^2([0,T];U)}^2\\
&=\int_0^T\|\1_{[r+\epsilon,T]}(s)B^\ast T^\ast (s-r-\epsilon)g(s)- \1_{[r,T]}(s)B^\ast T^\ast (s-r)g(s)\|^2\ud s\\
&=\int_{r+\epsilon}^T \| B^\ast T^\ast(s-r-\epsilon)(\text{Id}- T^\ast(\epsilon))g(s)\|^2 \ud s
+\int_r^{r+\epsilon} \| B^\ast T^\ast (s-r)g(s)\|^2\ud s.\\
&\leq m\int_{0}^T \|(\text{Id}- T^\ast(\epsilon))g(s)\|^2 \ud s
+\epsilon m \|g\|_{\infty}^2\\
&\to 0 \qquad\text{as $\epsilon\to 0$,}
\end{align*}
which shows that the mapping $r\mapsto f(\cdot,r)$ is right continuous. In a similar way, we establish that $r \mapsto f(\cdot,r)$ is left continuous. Thus, we can apply 
Theorem \ref{SFT} to conclude by using \eqref{eq.convolution-g} that 
the mapping $s\mapsto \scapro{Y(s)}{g(s)}$ is square-integrable on $[0,T]$. 

Let $(g_n)_{n\in\mathbb{N}}$ be a sequence in $C([0,T];V)$ with $g_n \to 0$. By Lemma 5.4 in \cite{OU} and Theorem \ref{SFT}, the L{\`e}vy symbol of the infinitely divisible random variable $\int_0^T \langle Y(s), g_n(s) \rangle \ud s$   is given by
\[ \Phi_n\colon \mathbb{R}\rightarrow \mathbb{C}, \qquad \Phi_n(\beta) = \int_0^T  \Psi \left(\int_r^T\beta B^*T^*(s-r)g_n(s)\ud s\right)\ud r,\]
where $\Psi\colon U \to\mathbb{C}$ is the L\'evy symbol of $L$. 
As $\Psi$ is continuous and maps bounded sets to bounded sets according to Lemma~3.2 and Lemma~5.1 of \cite{OU}, a repeated application of Lebesgue's theorem 
implies $\Phi_n(\beta) \rightarrow 0$ for every $\beta \in \mathbb{R}$, which proves that $Y$ is weakly Bochner regular.

Taking $T=t$ and $g(s)=A^\ast v$ for every $s\in [0,t]$ in the definition of $f$ in \eqref{f_alpha}, we can apply Theorem \ref{SFT} to obtain for each $v\in \mathcal{D}(A^*)$  that
	\begin{align*}
	\int_0^t \langle Y(s),A^*v\rangle ds & = \int_0^t\left(\int_0^sB^*T^*(s-r)A^*v\ud L(r)\right)\ud s\\
	& = \int_0^t\left(\int_r^tB^*T^*(s-r)A^*v\ud s\right)\ud L(r)\\
	&= \int_0^t\left(B^*T^*(t-r)v-B^*T^*(0)v\right)\ud L(r)\\
	&=\langle Y(t),v\rangle-L(t)(B^*v),
	\end{align*}
which shows \eqref{eq.def-weak-eq}.  Theorem \ref{th.stoch_cont} guarantees\footnote{although we present Theorem \ref{th.stoch_cont} later its proof is independent of any of the previous arguments.} that the stochastic process $\big(\int_0^t T(t-r)B\ud L(r):\, t\in[0,T]\big)$ is stochastically continuous and since it is also adapted, it has a progressively measurable modification by    Proposition 3.6 in \cite{daprato_zab} which completes the proof.
\end{proof}

To prove uniqueness of the solution we follow the same approach as in \cite{daprato_zab}, for which we need the following integration by parts formula.
\begin{lemma}\label{tau} If $g\colon [0,T]\rightarrow U$ is a function of the form $g(t) = \tau(t)u$ for $u \in U$ and $\tau\in C^1\left([0,T];\mathbb{R}\right)$,
 then
	\[\int_0^Tg(s)\ud L(s) = -\int_{0}^{T}L(s)(g'(s))\ud s + L(T)(g(T)). \]
\end{lemma}

\begin{proof}
For a sequence  $\{(t_k^n)_{k=0}^{N_n}:\, n\in\N\}$  of partitions of the interval $[0,T]$
with $ \max_{0\leq k \leq N_n-1}|t^n_{k+1}-t^n_{k}|\rightarrow 0$ as $n\rightarrow \infty$
define the simple functions
\begin{align*}
g_n\colon [0,T]\to U,\qquad
g_n(t):=\sum\limits_{k=0}^{N_n-1}g(t^n_k)\1_{[t^n_k,t^n_{k+1})}(t)
+\1_{\{T\}}(t)g(T).
\end{align*}
As $g_n$ converges to $g$ uniformly on $[0,T]$, Lemma~5.1 of \cite{OU} implies
\begin{align}\label{eq.g-limit-1}
\int_0^T g_n(s) \ud L(s)\to \int_0^T g(s)\ud L(s)
\qquad\text{in probability.}
\end{align}
On the other hand, $P$-almost surely we obtain
\begin{align}\label{eq.g-limit-2}
\int_0^T g_n(s)\ud L(s)
&=\sum\limits_{k=0}^{N_n-1}\left(L(t^n_{k+1})-L(t^n_k)\right)(\tau(t^n_k)u)\notag\\
& = \sum\limits_{k=0}^{N_n-1}\tau(t^n_k) \left(L(t^n_{k+1})-L(t^n_k)\right)(u)\notag\\
& = -\sum\limits_{k=0}^{N_n-1} \big(\tau(t^n_{k+1})-\tau(t^n_k)\big) L(t^n_{k+1})(u) + \tau(T) L(T)(u).
\end{align}
Applying the mean value theorem, we obtain for some $\xi_k^n \in (t^n_k, t^n_{k+1})$ that
	\begin{align}\label{eq.mean-value}
&\sum\limits_{k=0}^{N_n-1}\left(\tau(t^n_{k+1})-\tau(t^n_k)\right) L(t^n_{k+1})(u)\notag\\
&=	\sum\limits_{k=0}^{N_n-1} \tau'(\xi_k^n )(t^n_{k+1}-t^n_k) L(t^n_{k+1})(u)\notag\\
	&=\sum\limits_{k=0}^{N_n-1}\tau'(\xi_k^n )(t^n_{k+1}-t^n_k) L(\xi_{k}^n)(u)
 -\sum\limits_{k=0}^{N_n-1} \tau'(\xi_k^n)(t^n_{k+1}-t^n_k) \big(L(\xi_k^n)(u)-L(t^n_{k+1})(u)\big).
	\end{align}
As the map $s\mapsto \tau'(s) L(s)u $ has only countable number of discontinuities,  it is Riemann integrable and we obtain
\begin{equation}\label{eq20}
	\lim_{n\rightarrow\infty}\sum\limits_{k=0}^{N_n-1}\tau'(\xi_k^n )(t^n_{k+1}-t^n_k) L(\xi_{k}^n)(u)=\int_{0}^{T}L(s)(u\tau'(s))\ud s.
	\end{equation}
To show that the second term in  \eqref{eq.mean-value} approaches $0$ we  define
\begin{align*}
M^n_k :=\sup_{s\in[t^n_k,t^n_{k+1}]} L(s)u, \qquad\qquad  m^n_k:=\inf_{s\in[t^n_k,t^n_{k+1}]}L(s)u. 
\end{align*}
	Riemann integrability of the map $s\mapsto L(s)u$ implies
	\begin{align*}
&\left|\sum\limits_{k=0}^{N_n-1} \tau'(\xi_k^n)(t^n_{k+1}-t^n_k) \big(L(\xi_k^n)(u)-L(t^n_{k+1})(u)\big)\right|\\
&\qquad  \le \sum\limits_{k=0}^{N_n-1} \abs{\tau'(\xi_k^n)}\abs{t^n_{k+1}-t^n_k} \big|L(\xi_k^n)(u)-L(t^n_{k+1})(u)\big|\\
	& \qquad \leq \|\tau'\|_{\infty}\sum\limits_{k=0}^{N_n-1}\abs{t^n_{k+1}-t^n_k} \abs{M^n_k-m^n_k}\\
	& \qquad \to 0 \qquad \text{as}\quad n \to \infty. \numberthis \label{eq30}
	\end{align*}
Taking the limit in  \eqref{eq.g-limit-2} by applying \eqref{eq.mean-value}, \eqref{eq20} and  \eqref{eq30}  and comparing it to the limit in \eqref{eq.g-limit-1} completes the proof. 
\end{proof}

\begin{thm}\label{th.solution-implies-int}
	 If there exists a weak solution $Y$ of the stochastic Cauchy problem \eqref{SCP} then the mapping $s\mapsto T(s)B$ is stochastically integrable on $[0,T]$ with respect to $L$ and $Y$ is given by
\begin{align*}
Y(t)=T(t)y_0+ \int_0^t T(t-s)B \ud L(s).
\end{align*}
\end{thm}
\begin{proof}
	We can assume that $y_0 = 0$ due to linearity. For every $v\in \mathcal{D}(A^*) $ and $t\in[0,T]$ we have $P$-a.s. that
	\begin{equation}\label{eq4}
	\langle Y(t), v\rangle = \int_0^t \langle Y(s),A^*v\rangle \ud s+L(t)(B^*v).
	\end{equation}
	Let $f$ be in $C^1([0,T];\mathbb{R})$ and $v$ in $\mathcal{D}(A^*)$. By using \eqref{eq4} 
and applying the integration by parts formula in  Lemma \ref{tau} to $g(\cdot) = f(\cdot)B^*v$ and the classical integration by parts formula for Lebesgue integrals 
we obtain 
\begin{align*}
	\int_0^t f'(s)\langle Y(s), v\rangle \ud s
	& =\int_0^t f'(s)\left(\int_0^s\langle Y(r), A^*v\rangle \ud r\right)\ud s+\int_0^tf'(s)L(s)(B^*v)\ud s\\
	& = f(t)\int_0^t\langle Y(s), A^*v\rangle \ud s -  \int_0^t f(s) \langle Y(s), A^*v\rangle \ud s\\
	& \qquad + f(t)L(t)(B^*v)-\int_0^tf(s)B^*v \ud L(s). 
	\end{align*}
Rearranging the terms and using \eqref{eq4}, we obtain by defining $F(\cdot)=f(\cdot) v$ that
	\begin{equation}\label{eq5}
	\langle Y(t), F(t)\rangle = \int_0^t \langle Y(s),F'(s)+A^*F(s)\rangle \ud s+\int_0^t B^*F(s)\ud L(s).
	\end{equation}
	For $v\in \mathcal{D}(A^{*2})$, the function $G:=T^*(t-\cdot)v$ is in $C^1([0,t];\mathcal{D}(A^*))$.
Due to Lemma 8.4 in \cite{neerven}, we can find a sequence $F_n\in \mathrm{span}\{f(\cdot) w:f\in C^1([0,t];\mathbb{R}),\,w\in \mathcal{D}(A^*)\}$ such that $F_n$ converges to $G$ in $C^1([0,t];\mathcal{D}(A^*))$. Then $F'_n+A^*F_n \to 0$ in $C([0,t];V)$.  The weakly Bochner regularity implies for a subsequence that
	\[ \int_0^t \langle Y(s),F'_{n_k}(s)+A^*F_{n_k}(s)\rangle \ud s \to  0 \quad
	\text{$P$-a.s.}\]
	Moreover, since $B^*F_n$ converges to $B^*G$ in $C([0,t];U)$,  Lemma~5.2 in \cite{OU} implies
	\[\int_0^t B^*F_n(s)\ud L(s)\rightarrow \int_0^t B^*G(s)\ud L(s)\qquad \mathrm{in\;probability}.\]
	Consequently, \eqref{eq5} holds for $F$ replaced by $G$, which gives
	\[\langle Y(t), v\rangle = \int_0^t B^*T^*(t-s)v\ud L(s) \qquad\text{for all } v \in \mathcal{D}(A^{*2}).\]
	 Since $\mathcal{D}(A^{*2})$ is dense in $V$, for any $v \in V$, we can find a sequence $\{v_n\}$ in $\mathcal{D}(A^{*2})$ with $v_n \rightarrow v$ as $n \rightarrow \infty$.
	Since $B^*T^*(t-\cdot)v_n$ converges to $ B^*T^*(t-\cdot)v$ in $C([0,t];U)$ it follows from \cite[Lemma 5.2]{OU} that
	\[\lim \limits_{n\rightarrow \infty}\int_0^t B^*T^*(t-s)v_n\ud L(s) = \int_0^t B^*T^*(t-s)v\ud L(s) \qquad \mathrm{in\; probability},\]
	and hence $P$-a.s.
	\[\langle Y(t), v\rangle = \int_0^t B^*T^*(t-s)v\ud L(s)  \qquad \mathrm{for\; all\;}  v\in V.\]
	 This establishes the stochastic integrability of $s \mapsto T(s)B$ on $[0,T]$.
\end{proof}

\section{Properties of the solution} \label{se.properties}

We begin this section with  discussing some path properties of the solution. 
Various specific examples of the stochastic Cauchy problem 
\eqref{SCP} were observed in the literature in which the solution $Y$ exists but does not have a modification 
$\tilde{Y}$  with scalarly c\`adl\`ag paths; see e.g.\ \cite{Brzetal}, \cite{LiuZhai}
and \cite{PeszatZab12}.
Even the weaker property that the real valued process $(\scapro{Y(t)}{v}:\, t\in [0,T])$ has a modification with  c\`adl\`ag paths
for each $v\in V$ can be verified only in a few specific examples. However, our stochastic Fubini Theorem~\ref{SFT} immediately implies that this real valued 
stochastic process $(\scapro{Y(t)}{v}:\, t\in [0,T])$ has square-integrable trajectories:
\begin{thm}
	If $(Y(t): t\in [0,T])$ is the weak solution of the stochastic Cauchy problem \eqref{SCP}, then for every $v \in V$, $P$-a.s.
\begin{align*}
	\int_0^T\langle Y(t),v\rangle^2 \ud t < \infty.
\end{align*}
\end{thm}
\begin{proof}
By choosing $g(s)=v$ for all $s\in [0,T]$ in \eqref{f_alpha}, the following arguments in the proof of Theorem \ref{existence} show
that the function 
\begin{align*}
f\colon [0,T]\times [0,T]\to U, 
\qquad f(s,r)=\1_{[0,s]}(r)B^\ast T^\ast(s-r)v.
\end{align*}
satisfies the assumption of Theorem \ref{SFT}. Consequently, we conclude that
the stochastic process $(\scapro{Y(t)}{v}:\, t\in [0,T])$ defines a random variable in $L^2([0,T];\R)$ for each $v\in V$.
\end{proof}

\begin{thm}\label{th.stoch_cont}
	The weak solution $(Y(t): t\in[0,T])$ of the stochastic Cauchy problem \eqref{SCP} is stochastically continuous.
\end{thm}
\begin{proof}
	We can assume that $y_0 =0$. Theorem \ref{th.solution-implies-int} implies 
\begin{align*}	
	Y(t)  =\int_0^tT(t-s)B\ud L(s)
	\qquad\text{for all }t\in [0,T].
\end{align*}
Let $P_t$ denote the probability distribution of $Y(t)$. 
By \cite[Lemma 2.4]{jaku1}, it is enough to show that
	\begin{enumerate}[(i)]
		\item   $\big(\scapro{Y(t)}{v}:\, t\in [0,T]\big)$ is stochastically continuous for each $v \in V$;
		\item $\{P_t:\, t\in [0,T]\}$ is relatively compact in ${\mathcal M}(V)$.
	\end{enumerate}
Proof of (i): for every $t\in [0,T]$,  $v\in V$ and  $\epsilon \ge 0$, we have by 
\eqref{eq.change-int-product} that
	\begin{align*}
&	\left|\scapro{Y(t+\epsilon)}{v}-\scapro{Y(t)}{v}\right|\\
&\quad = \left|\int_0^{t+\epsilon}B^*T^*(t+\epsilon-s)v\ud L(s)-\int_0^tB^*T^*(t-s)v\ud L(s)\right|\\
&\quad \leq \left|\int_0^{t} B^*T^*(t-s)(T^*(\epsilon)v-v)\ud L(s)\right|+\left|\int_t^{t+\epsilon}B^*T^*(t+\epsilon-s)v\ud L(s)\right|.\numberthis \label{rcrsc}
	\end{align*}
Define the random variables 
\begin{align*}
I_1(\epsilon):=\int_0^{t} B^*T^*(t-s)(T^*(\epsilon)v-v)\ud L(s),
\quad
I_2(\epsilon):=\int_t^{t+\epsilon}B^*T^*(t+\epsilon-s)v\ud L(s).
\end{align*}	
The random variable $I_1(\epsilon)$ has the characteristic function $\phi_{1,\epsilon}\colon\R\to\C$ given by
\begin{align*}
\phi_{1,\epsilon}(\beta)= \exp\left(\int_0^{t}\Psi\big(\beta B^*T^*(s)(T^*(\epsilon)v-v)\big)\ud s   \right) .
\end{align*}
By using standard properties of the semigroup we obtain 
\begin{align*}
	\sup_{s\in [0,T]}\|\beta B^*T^*(s)(T^*(\epsilon)v-v)\|
\to 0\qquad\text{as }\epsilon \to 0,
\end{align*}
which implies $\phi_{1,\epsilon}(\beta)\to 1$ for all $\beta\in\R$ due to Lemma 5.1 in \cite{OU}. Thus, $I_1(\epsilon)$ converges to $0$ in probability as $\epsilon \to 0$.
The characteristic function $\phi_{2,\epsilon}\colon\R\to\C$ of the 
random variable $I_2(\epsilon)$ obeys
	\begin{align*}
\phi_{2,\epsilon}(\beta)  
= \exp \left(\int_0^{\epsilon}\Psi(\beta B^*T^*(s)v)\ud s\right)
	 \rightarrow 1 \text{   as } \epsilon \rightarrow 0.
	\end{align*}
Consequently, we obtain that $I_2(\epsilon)\to 0$ in probability. The arguments above show by \eqref{rcrsc} that 
$\scapro{Y(t+\epsilon)}{v}\to \scapro{Y(t)}{v}$ in probability as $\epsilon\to 0$. Analogously, we can show that $\scapro{Y(t-\epsilon)}{v}\to \scapro{Y(t)}{v}$ in probability, which yields Property (i). 

Proof of (ii): it follows from Lemma 5.4 in \cite{OU} that the probability distribution $P_t$ of $Y(t)$ is an infinitely divisible probability measure in ${\mathcal M}(V)$ with characteristics $(c_t,S_t,\theta_t)$ given for all $v\in V$ by
\begin{align*}
\langle c_t,v\rangle &=\int_0^t a(B^*T^*(s)v)\ud s +\int_V\langle h, v\rangle \left(\1_{B_V}(h)-\1_{B_{\mathbb{R}}}(\langle h, v\rangle)\right)\theta_t(\udd h),\\
 \langle S_{t} v, v\rangle & = \int \limits_{0}^{t} \langle B^*T^*(s)v, QB^*T^*(s)v\rangle\ud s, \\
 \theta_t & = ( \text{leb}\otimes\mu )\circ \chi_t^{-1} \quad \text{on } \mathcal{Z}(V), 
\end{align*}
where
$\chi_t \colon [0,t]\times U \rightarrow V$ is defined by $\chi_t(s,u):=T(s)Bu$. 

Let $\tilde{P}_t$ denote the infinitely divisible probability measure 
with characteristics $(0,S_t,\theta_t)$. Theorem~VI.5.1 in \cite{partha} guarantees that the set $\{\tilde{P}_t:\, t\in [0,T]\}$  is relatively compact if and only if
 the set $\{\theta_t:\,t\in [0,T]\}$ restricted to the complement of any neighbourhood of the origin is relatively compact in 
 ${\mathcal M}(V)$ and the operators $T_t\colon V \rightarrow V$ defined by
\begin{align*}
\langle T_tv,v\rangle :=\langle S_tv,v\rangle+\int_{\|h\| \leq 1}\langle v,h\rangle^2\, \theta_t(\udd h) 
\end{align*}
satisfy
\begin{align}\label{eq.cond-T}
\sup \limits_{t\in [0,T]} \sum \limits_{k=1}^{\infty}\langle T_th_k,h_k\rangle <\infty
\qquad\text{and}\qquad
\lim \limits_{N \rightarrow \infty}\sup \limits_{t\in [0,T]} \sum \limits_{k=N}^{\infty}\langle T_th_k,h_k\rangle =0.
\end{align}
For a set $A$ in the cylindrical algebra $\mathcal{Z}(V)$ we have
\begin{align*}
\theta_t(A)= \int_{0}^t\int_U\1_{A}(T(s)Bu)\mu(\udd u)\ud s
\le \int_{0}^T\int_U\1_{A}(T(s)Bu)\mu(\udd u)\ud s
 = \theta_T(A).
\end{align*}
 Since  $\Borel(V)$ is the sigma algebra generated by $\mathcal{Z}(V)$ and $\mathcal{Z}(V)$ is closed under intersection,  we conclude $\theta_t \leq \theta_T$ on $\Borel(V)$ for all $t \in [0,T]$. Let $\theta_t^{c}$ denote the restriction 
 of $\theta_t$ to the complement of a neighbourhood of the origin.
  Since $\theta_T^{c}$ is a Radon measure by  \cite[Prop 1.1.3]{linde},  there exists for each $\epsilon >0$ a compact set $K \subseteq V$ such that $\theta_T^{c}(K^c) \leq \epsilon$. Consequently, we obtain $\theta_t^{c}(K^c) \leq \theta_T^{c}(K^c) \leq \epsilon$ for all  $t \in [0,T]$, which shows by Prokhorov's theorem that 
 $\{\theta_t: t \in [0,T]\}$ restricted to the complement of any neighbourhood of the origin is relatively compact in ${\mathcal M}(V)$. 
 
 The stochastic integrability of $s \mapsto T(s)B$ implies by \eqref{sicond2} 
and Lebesgue's theorem that
 \begin{align}
\lim_{N\to\infty}\sup \limits_{t\in [0,T]} \sum_{k=N}^{\infty}\langle S_th_k, h_k
\rangle\ud s =\lim_{N\to\infty} \int_0^T \sum \limits_{k=N}^{\infty}\langle T(s)BQB^*T^*(s)h_k,h_k\rangle\ud s =0.\label{Ttcondb1}
   \end{align}
 Condition \eqref{sicond3} of stochastic integrability implies
   \begin{align*}
&\sup \limits_{t\in [0,T]} \sum \limits_{k=N}^{\infty}\int_{\|h\| \leq 1}\langle h_k,h\rangle^2 \theta_t(\udd h)\\
&\qquad\qquad  \leq \sup \limits_{t\in [0,T]}\, \sup \limits_{m \geq N}\int_{V}\left(\sum \limits_{k=N}^{m}\langle h_k,h\rangle^2 \wedge 1\right)\theta_t(\udd h)\\
  &\qquad\qquad  = \sup \limits_{t\in [0,T]}\, \sup \limits_{m \geq N}\int_0^t\int_{U}\left(\sum \limits_{k=N}^{m}\langle h_k,T(s)Bu\rangle^2 \wedge 1\right)\mu(\udd u)\ud s\\
  &\qquad\qquad  =\sup \limits_{m \geq N}\int_0^T\int_{U}\left(\sum \limits_{k=N}^{m}\langle h_k,T(s)Bu\rangle^2 \wedge 1\right)\mu(\udd u)\ud s\\
 &\qquad\qquad  \to 0 \qquad\text{as $N\to\infty$.}  \numberthis \label{Ttcondb2}
  \end{align*}
The limits  \eqref{Ttcondb1} and \eqref{Ttcondb2} show that the second
condition in \eqref{eq.cond-T} is satisfied. As the first condition in 
\eqref{eq.cond-T} follows analogously, we conclude that $\{\tilde{P_t}:\, t\in [0,T]\}$ is relatively compact. 

Let $\{\tilde{P}_{t_n}\}_{n\in\N}$ be a weakly convergent subsequence. 
Without any restriction we can assume that there exists $t\in [0,T]$ 
such that $t_n\to t$. For the characteristic functions $\phi_{P_{t_n}}$ of $P_{t_n}$  we obtain
\begin{align*}
&|\phi_{P_{t_n}}(v)-\phi_{P_{t}}(v)|\\
&\qquad=\left|\exp \left(\int_0^{t_n}\Psi(B^*T^*(t_n-s)v)\ud s\right)-\exp\left(\int_0^{t}\Psi(B^*T^*(t-s)v)\ud s\right)\right|\\
&\qquad =\left|\exp\left(\int_t^{t_n}\Psi(B^*T^*(s)v)\ud s\right)-1\right|\left|\exp\left(\int_0^{t}\Psi(B^*T^*(s)v)\ud s\right)\right|.\numberthis\label{cfofmut}
\end{align*}
Since $\psi$ maps bounded sets to bounded sets, we obtain for each $\delta>0$ that
\[\sup \limits_{\|v\|<\delta}\left|\int_t^{t_n}\Psi(B^*T^*(s)v)\ud s\right|\rightarrow 0 \quad \text{as  } n\rightarrow \infty,\]
which implies by \eqref{cfofmut} that
\begin{align*}
\sup \limits_{\|v\|<\delta}|\phi_{P_{t_n}}(v)-\phi_{P_t}(v)|  \rightarrow 0 \quad \text{as  } n\rightarrow \infty. 
\end{align*}
As $\tilde{P}_{t_n}=P_{t_n}\ast \delta_{-c_{t_n}}$, Theorem~2.3.8 in \cite{linde} implies that $\{P_{t_n}\}$ converges weakly, which completes the proof of Property (ii).
\end{proof}	

As mentioned in the introduction, it has been observed for specific examples of a cylindrical L\'evy process, that the solution of \eqref{SCP} has 
highly irregular paths in an analytical sense. In our general setting, we state a condition in the result below which implies such highly irregular paths of the solution. 
This condition does not only allow a geometric interpretation of this phenomena but is also easy to  verify in many examples including the ones considered in the literature. 
\begin{thm}\label{th.nocadlagcond}
Assume that an orthonormal basis $(h_k)_{k\in\N}$ of $V$ is in $\mathrm{Dom}(A^\ast)$ and let $L$ be a cylindrical L{\'e}vy process  with cylindrical characteristics  $(a,Q,\mu)$.
If there exists a constant $c>0$ such that
\begin{align}\label{no-cadlag.condition}
 \lim_{n\to\infty} \mu\left(\left\{u\in U: \sum_{k=1}^n \scapro{u}{B^\ast h_k}^2 >c \right\}\right)=\infty,
\end{align}
then there does not exist any modification $\widetilde{Y}$ of the
weak solution $Y$ of \eqref{SCP} such that for each $v\in V$
the stochastic process $(\scapro{\widetilde{Y}(t)}{v}:\, t\in [0,T])$
has c{\`a}dl{\`a}g paths.
\end{thm}
\begin{remark}
Note, that if $\mu$ is a genuine L\'evy measure then Condition \eqref{no-cadlag.condition} cannot  be satisfied for any constant $c>0$. This is due to the fact that in this case, $\mu$ is a finite Radon measure on each complement of the origin; see \cite{linde}.

\end{remark}
\begin{example}(continues Example \ref{ex.series-existence}).
Assume that the cylindrical L\'evy process $L$ is given by \eqref{example.series} and $B=\text{Id}$ in equation \eqref{SCP}.
The independence of the real valued L\'evy processes $(\ell_k)_{k\in \N}$ implies that the cylindrical
L\'evy measure $\mu$ has support only in $\cup_
{k=1}^{\infty}\text{span}\{e_k\}$, and thus Condition \eqref{no-cadlag.condition}
 reduces to
\begin{align*}
 \sum_{k=1}^{\infty} \mu\Big(\Big\{u\in U:  \scapro{u}{ h_k}^2 >c \Big\}\Big)=\infty,
\end{align*}
for a constant $c > 0$. For this special case, the conclusion of 
Theorem \ref{th.nocadlagcond} has already been derived in \cite{PeszatZab12}.

For example, if $(\ell_k)_{k\in \N}$ is a family of independent, identically distributed
 symmetric $\alpha$-stable L\'evy processes, then Condition 
 \eqref{no-cadlag.condition} is satisfied for $B=\text{Id}$; see \cite{LiuZhai}.
\end{example}

\begin{example}(continues Example \ref{ex.canonical-stable-existence}). 
Let $L$ be the canonical $\alpha$-stable process, introduced in Example \ref{ex.canonical-stable-existence}. By using properties 
of $\alpha$-stable  distributions in $\R^n$ one calculates for each $n\in\N$ that
\begin{align*}
\mu\left(\left\{u\in U: \sum_{k=1}^n \scapro{u}{ h_k}^2 >c \right\}\right) = \frac{1}{c^{\alpha} c_{\alpha}}\frac{\Gamma\left(\frac{1}{2}\right)\Gamma\left(\frac{n+\alpha}{2}\right)}{\Gamma\left(\frac{n}{2}\right)\Gamma\left(\frac{1+\alpha}{2}\right)},
\end{align*}
where $\Gamma$ denotes the Gamma function and $c_\alpha$ is a constant only depending on $\alpha$. As the right hand side converges to $\infty$ as $n\to  \infty$, Condition \eqref{no-cadlag.condition} is satisfied for $B=\text{Id}$; see \cite[Theorem 5.1]{Riedle_alpha_stable}.
\end{example}

\begin{proof}[Proof of Theorem \ref{th.nocadlagcond}]
(The proof is based on ideas from \cite{LiuZhai}).
For every $n\in \N$ and $t\in [0,T]$  define the random vectors $L_n(t):=\big(L(t)B^\ast h_1,\dots, L(t)B^\ast h_n\big)$ and $Y_n(t):=\big(\scapro{Y(t)}{h_1},\dots, \scapro{Y(t)}{h_n}\big)$. It follows from Definition \ref{de.weak-solution} of a weak solution that for every $t\in [0,T]$ we have $P$-a.s
\begin{align*}
  Y_n(t)=Y_n(0)+ \int_0^t \big(\scapro{Y(s)}{A^\ast h_1}, \dots, \scapro{Y(s)}{A^\ast h_n}\big)\,ds + L_n(t).
\end{align*}
Consequently, the $n$-dimensional processes $(Y_n(t):\,t\in [0,T])$ and $(L_n(t):\,t\in [0,T])$  jump at the same time by the same size, which implies
\begin{align*}
  \sup_{t\in [0,T]}\abs{\Delta L_n(t)}^2
  =  \sup_{t\in [0,T]}\abs{\Delta Y_n(t)}^2
  \le 4 \sup_{t\in [0,T]}\abs{Y_n(t)}^2,
\end{align*}
where $\Delta g(t):=g(t)-g(t-)$ for c{\`a}dl{\`a}g functions $g\colon [0,T]\to \R^n$.
It follows that
\begin{align*}
 P\left(\sup_{t\in [0,T]} \sum_{k=1}^\infty \scapro{Y(t)}{h_k}^2<\infty\right)
 &=\lim_{c\to\infty}   P\left(\sup_{n\in\N}\sup_{t\in [0,T]} \sum_{k=1}^n \scapro{Y(t)}{h_k}^2\le\frac{1}{4}c^2\right)\notag\\
&=\lim_{c\to\infty} \lim_{n\to\infty}  P\left(\sup_{t\in [0,T]} \sum_{k=1}^n \scapro{Y(t)}{h_k}^2\le\frac{1}{4}c^2\right)\notag\\
&=\lim_{c\to\infty} \lim_{n\to\infty}  P\left(\sup_{t\in [0,T]} \abs{Y_n(t)}^2\le\frac{1}{4}c^2\right)\notag\\
&\le\lim_{c\to\infty} \lim_{n\to\infty}  P\left(\sup_{t\in [0,T]} \abs{\Delta L_n(t)}^2\le c^2\right)\notag\\
&=\lim_{c\to\infty} \lim_{n\to\infty} \exp\left(-T \mu_n\Big(\{\beta\in\R^n:\abs{\beta}>c\}\Big)\right),
\end{align*}
where $\mu_n$ denotes the L{\'e}vy measure of the $\R^n$-valued L{\'e}vy process $L_n$. Since
$\mu_n=\mu\circ\pi_{n}^{-1}$ for $\pi_n\colon U\to \R^n$ 
and $\pi_n u=(\scapro{u}{B^\ast h_1)},\dots, \scapro{u}{B^\ast h_n})$  due to \cite[Th. 2.4]{app}, we obtain by \eqref{no-cadlag.condition} that
\begin{align*}
 P\left(\sup_{t\in [0,T]} \sum_{k=1}^\infty \scapro{Y(t)}{h_k}^2<\infty\right)=0,
\end{align*}
which completes the proof by an application of Theorem 2.3 in \cite{PeszatZab12}.
\end{proof}

We continue to consider mean square continuity of the solution. For this purpose, we naturally require that the cylindrical L\'evy process has weak second moments, i.e.\ $E[|L(1)u|^2] < \infty$ for all $u \in U$.
In this case, the cylindrical L\'evy process with characteristics $(a,Q,\mu)$  can 
be written as 
\begin{align*}
L(t)u=t\scapro{\tilde{a}}{u} + W(t)u + M(t)u \qquad\text{for all }
t\ge 0,\, u\in U,
\end{align*}
where $\tilde{a}\in U$, $W$ is a cylindrical Brownian motion with covariance operator $Q$ and $M$ is a cylindrical L\'evy process independent of $W$ and with characteristics $(a^\prime,0,\mu)$. Here $a^\prime\colon U\to \R$ is defined by $a^\prime(u):=-\int_{\abs{\beta}>1}\beta \, (\mu\circ  u^{-1})(\udd \beta)$ and $\scapro{\tilde{a}}{u}=a(u)-a^\prime(u)$
 for all $u\in U$; see Corollary 3.12 in \cite{app}. It follows for any function $f\in R([0,T];U)$ that 
 \begin{align}\label{eq.int-decomp}
 \int_0^t f(s)\ud L(s)=
 \int_0^t \scapro{\tilde{a}}{f(s)}\ud s
  +\int_0^t f(s)\ud W(s) + \int_0^t f(s)\ud M(s).
 \end{align}

\begin{example}
Assume that $L$ has weak second moments. If 
\begin{align}\label{sibmcond}
\int_0^T\|T(s)B\|^2_{\text{HS}}\ud s < \infty,
\end{align}
then there exists a weak solution $(Y(t):\, t\in [0,T])$ of the Cauchy problem \eqref{SCP} and it satisfies $E[\|Y(t)\|^2]<\infty$ for all $t\in [0,T]$.

\begin{proof}
For showing the existence of a solution, we have to establish that $t\mapsto T(t)B$ is stochastically integrable.  Conditions \eqref{sicond12} and \eqref{sicond2} can be verified similarly as in the proof of Lemma \ref{phistochint}.
Since $L$ has weak second moments, the closed graph theorem guarantees that $L(t)\colon U \rightarrow L^2_{P}(\Omega;\mathbb{R})$ is continuous, which  implies 
\begin{align*}
C:=\sup \limits_{\|u^*\| \leq 1}\int_U \langle u, u^*\rangle^2\mu(\udd u)
\leq\|L(1)\|^2_{\text{op}} <\infty.
\end{align*}
Consequently,  Condition \eqref{sicond3} is satisfied since 
\begin{align}\label{eq.T-HS-Int-2}
&\int_0^T\int_U  \left(\sum_{k=m}^{n }\langle u, B^*T^*(s)h_k \rangle^2\wedge 1\right)\mu(\udd u)  \ud s \notag \\
&\qquad\le \sum_{k=m}^{n}\int_0^T\int_U  \langle u, B^*T^*(s)h_k \rangle^2\mu(\udd u)  \ud s \notag \\
&\qquad  \le \sum_{k=m}^{n}\int_0^T\int_U\|B^*T^*(s)h_k\|^2 \left \langle u,     
    \frac{B^\ast T^*(s)h_k}{\|B^*T^*(s)h_k\|} \right \rangle^2\mu(\udd u) \ud s\notag \\
& \qquad \le \sup _{\|u^*\| \leq 1}\int_U \langle u, u^* \rangle^2\, \mu(\udd u)\sum_{k=m}^{n}\int_0^T\|B^*T^*(s)h_k\|^2 \ud s  \notag \\
&\qquad \to 0\qquad\text{as }m,n\to\infty, 
\end{align}
where we applied \eqref{sibmcond} in the last line.
As the L\'evy measure $\theta_t$ 
of the infinitely divisible random variable $Y(t)$ is given by $( \text{leb}\otimes \mu ) \circ \chi_t^{-1}$ on $\mathcal{Z}(V)$ where $\chi_t\colon[0,t] \times U \rightarrow V$
and $\chi_t(s,u)=T(s)Bu$, we obtain by a similar calculation as in  \eqref{eq.T-HS-Int-2}  that  
\begin{align*}
\int_V \|v\|^2\, \ud \theta_t(v)=
\sum_{k=1}^{\infty}\int_0^t\int_U \langle u, B^*T^*(s)h_k \rangle^2\mu(\udd u)  \ud s
\le  C\int_0^t\|B^*T^*(s)\|^2_{\text{HS}}\ud s
 < \infty.
\end{align*}
Consequently, we have $E[\|Y(t)\|^2]<\infty$ for all $t\in [0,T]$.
\end{proof}
\end{example}

\begin{thm} Assume that $L$ has weak second moments.
If  the weak  solution $(Y(t): t\in [0,T])$ of the stochastic Cauchy problem \eqref{SCP} has finite second moments, i.e.\ $E[\|Y(t)\|^2]<\infty$ for all $t\in [0,T]$, 
then $Y$ is continuous in mean-square, i.e.\  $Y\in C([0,T];L^2_P(\Omega; V))$.
\end{thm}
\begin{proof}
Let $\Phi\colon [0,T]\to {\mathcal L}(U,V)$ be a  stochastically integrable,  regulated function and $\Phi(\cdot)\tilde{a}$ be Pettis integrable. Then we obtain for each $t\in [0,T]$ and $\Psi\in {\mathcal L}(V,V)$ by \eqref{eq.int-decomp} 
and using the fact that $W$ and $M$ have mean zero and are independent:  
\begin{align}\label{eq.mean-square-formula}
&E\left[\left\|\int_0^t \Psi\Phi(t-s) \ud L(s)\right\|^2\right]\notag\\
&\qquad\qquad =\sum_{k=1}^\infty E\left[ \left|\int_0^t  \Phi^\ast(t-s)\Psi^\ast h_k\ud L(s)\right|^2\right]\notag\\
&\qquad\qquad = \sum_{k=1}^\infty \Bigg( E\left[\left| \int_0^t \scapro{\tilde{a}}{\Phi^\ast(s)\Psi^\ast h_k}\ud s \right|^2\right]
+ \int_0^t \scapro{Q\Phi^\ast(s)\Psi^\ast h_k}{\Phi^\ast (s)\Psi^\ast h_k}\ud s \notag\\
 &\qquad\qquad\qquad \qquad + \int_0^t \int_U \scapro{u}{\Phi^\ast (s)\Psi^\ast h_k }^2\, \mu(\udd u)\ud s\Bigg)\notag\\
&\qquad\qquad= \left\| \int_0^t \Psi\Phi(s)\tilde{a}\ud s\right\|^2 
 + \int_0^t \Big\|\Psi \Phi(s) Q^{1/2}\Big\|^2_{\text{HS}}\ud s 
 + \int_V \|\Psi v\|^2 \,\eta_t(\mathrm{d}v),
\end{align}
where $\eta_t$ is the (genuine) L\'evy measure of $\int_0^t\Phi(s)\ud L(s)$ and  is given by $\eta_t=( \text{leb}\otimes \mu ) \circ \xi_t^{-1}$ where $\xi_t\colon[0,t] \times U \rightarrow V$ is defined by  $\xi_t(s,u)=\Phi(s)u$. 

 We can assume $y_0=0$. Theorem \ref{th.solution-implies-int} implies 
\begin{align*}
 Y(t) = \int_0^tT(t-s)B\ud L(s) \qquad\text{for all }t\in [0,T].
 \end{align*}
 As $Y(t)$ has finite second moments it follows
 $\int_V \|v\|^2\,\theta_t(\udd v)<\infty$, where $\theta_t$ is the (genuine) L\'evy measure of $Y(t)$ and  is given by $\theta_t=( \text{leb}\otimes \mu ) \circ \chi_t^{-1}$ where $\chi_t\colon[0,t] \times U \rightarrow V$ is defined by  $\chi_t(s,u)=T(s)Bu$.
For any $t \in [0,T]$ and $\epsilon>0$ we obtain 
\begin{align}\label{eq.L2-main-inequality}
 &  E[\|Y(t+\epsilon)-Y(t)\|^2]\notag \\
 &= E\left[\left\|\int_0^{t}(T(t+\epsilon-s)B -T(t-s)B)\ud L(s)+\int_t^{t+\epsilon}T(t+\epsilon-s)B\ud L(s)\right\|^2\right]\notag \\
  & \le 2E\left[\left\|\int_0^{t}(T(\epsilon)-\I)T(t-s)B \ud L(s)\right\|^2\right]+2E\left[\left\|\int_t^{t+\epsilon}T(t+\epsilon-s)B\ud L(s)\right\|^2\right].
 \end{align}
By applying \eqref{eq.mean-square-formula}  we conclude
\begin{align*}
&E\left[\left\|\int_0^{t}(T(\epsilon)-\I)T(t-s)B \ud L(s)\right\|^2\right]\\
&\qquad\qquad \le t \int_0^t \|(T(\epsilon)-\I)T(s)B\tilde{a}\|^2\ud s
  +\int_0^t \left\|(T(\epsilon) -\I)T(s)BQ^{1/2}\right\|^2_{\text{HS}} \ud s\\
&\qquad \qquad\qquad\qquad  +\int_V \|(T(\epsilon)-\I)v\|^2\,\theta_t(\udd v).
\end{align*} 
Applying Lebesgue's theorem to each of the terms above shows
\begin{align}\label{eq.L2-aux1}
E\left[\left\|\int_0^{t}(T(\epsilon)-\I)T(t-s)B \ud L(s)\right\|^2\right]
\to 0\qquad\text{as $\epsilon\to 0$.}
\end{align}
By a similar computation as in \eqref{eq.mean-square-formula} we 
obtain for the second term in \eqref{eq.L2-main-inequality} that 
\begin{align}\label{eq.second-term-mean-square}
&E\left[\left\|\int_t^{t+\epsilon}T(t+\epsilon-s)B\ud L(s)\right\|^2\right]\notag\\
&\qquad\qquad =\left\| \int_0^\epsilon T(s)B\tilde{a}\ud s\right\|^2 +\int_0^\epsilon \left\|T(s)BQ^{1/2}\right\|^2_{\text{HS}} \ud s\notag\\
&\qquad\qquad\qquad\qquad + \sum \limits_{k=1}^{\infty}\int_0^{T}\1_{[0,\epsilon]}(s)\int_U\langle u, B^*T^*(s)h_k\rangle^2 \, \mu (\udd u)\ud s.
\end{align}
The first two terms in \eqref{eq.second-term-mean-square} converge to 0 as $\epsilon \to 0$. Since 
\begin{align*}
\sum \limits_{k=1}^{\infty}\int_0^{T}\1_{[0,\epsilon]}(s)\int_U\langle u, B^*T^*(s)h_k\rangle^2\mu (\udd u)\ud s
\le \int_V \|v\|^2\, \theta_T(\udd v)<\infty,
\end{align*}
we can apply Lebesgue's theorem to the third term in \eqref{eq.second-term-mean-square} and
obtain
\begin{align}\label{eq.L2-aux2}
E\left[\left\|\int_t^{t+\epsilon}T(t-s)B\ud L(s)\right\|^2\right]
\to 0 \qquad\text{as $\epsilon\to 0$.}
\end{align}
Applying \eqref{eq.L2-aux1} and \eqref{eq.L2-aux2} to \eqref{eq.L2-main-inequality} 
shows that $Y$ is mean-square continuous from the right. Analogously, we can prove that $Y$ is mean-square continuous from the left which completes the proof.
\end{proof}

We now discuss the flow property and Markov property of the solution of the stochastic Cauchy problem \eqref{SCP}.
For this purpose we assume that $t\mapsto T(t)B$ is stochastically integrable and define for $0\le s\le t \le T$ the mapping
\begin{align*}
\Phi_{s,t}\colon V\times \Omega \rightarrow V, 
\qquad
 \Phi_{s,t}(v)=T(t-s)v+\int_s^tT(t-r)B\ud L(r).
\end{align*}
\begin{thm}
Let $(Y(t):\, t\in [0,T])$ be the weak solution of \eqref{SCP}. Then we have:
\begin{enumerate}
\item[{\rm (a)}] the family $\{\Phi_{s,t}:\, 0\le s\le t\le T\}$ is 
a stochastic flow, i.e.\ $\Phi_{s,s}=\rm{Id}$ and
\begin{align*}
 \Phi_{s,t}\circ \Phi_{r,s}=\Phi_{r,t}
 \qquad \text{for all }0\le r\le s\le t\le T.
\end{align*}
\item[{\rm (b)}] the weak solution $(Y(t):\, t\in [0,T])$
is a Markov process with respect to 
the filtration $({\mathcal F}_t)_{t\in [0,T]}$ defined by 
$\mathcal F_t:=\sigma(\{L(s)u:\, s\in [0,t],\, u\in U\})$. 
\end{enumerate}
\end{thm}

\begin{proof} (a): we first show that for all $0\le r\le s\le t\le T$ we have 
\begin{align}\label{stoch_flow}
 T(t-s)\left(\int_r^sT(s-q)B\ud L(q)\right)=\int_r^sT(t-q)B\ud L(q).
\end{align}
For any $v \in V$, we obtain by \eqref{eq.change-int-product}
\begin{align*}
\left \langle T(t-s)\left(\int_r^sT(s-q)B\ud L(q)\right), v\right \rangle 
    &= \left \langle \int_r^sT(s-q)B\ud L(q), T^*(t-s)v\right \rangle\\
	&=\int_r^sB^*T^*(s-q)(T^*(t-s)v)\ud L(q)\\
	&=\int_r^sB^*T^*(t-q)v \ud L(q)\\
	&=\left \langle \int_r^sT(t-q)B\ud L(q), v\right \rangle,
\end{align*}
which shows \eqref{stoch_flow}. This enables us to conclude
\begin{align*}
	\Phi_{s,t}(\Phi_{r,s}(v))&=T(t-s)\Phi_{r,s}(v)+\int_s^tT(t-q)B\ud L(q)\\
	&=T(t-s)\left(T(s-r)v+\int_r^sT(s-q)B\ud L(q)\right)+\int_s^tT(t-q)B\ud L(q)\\
    &=T(t-r)v+\int_r^sT(t-q)B\ud L(q)+\int_s^tT(t-q)B\ud L(q)\\
	&=T(t-r)v+\int_r^tT(t-q)B\ud L(q)\\
	&=\Phi_{r,t}(v),
\end{align*}
which completes the proof of (a).

(b): by construction of stochastic integrals, we deduce that each $\Phi_{s,t}(v)$ is measurable with respect to $\sigma(\{L(q)u-L(p)u:\, s\le p< q\le t,\, u\in U\})$ for each $v\in V$. The independent increments of $L$ guarantee that $\Phi_{s,t}(v)$ is independent of $\mathcal{F}_s$. Consequently, by using 
Part (a) we obtain for any bounded, measurable function $f\colon V\to \R$ that
\begin{align*}
E\big[f(\Phi_{0,t+s}(y_0))|\mathcal{F}_s\big]= E\big[f(\Phi_{s,t+s}\circ \Phi_{0,s}(y_0))|\mathcal{F}_s\big]
=g_{s,t,f}(\Phi_{0,s}(y_0)), 
\end{align*}
where $g_{s,t,f}(v):=E\big[f(\Phi_{s,t+s}(v))\big]$ for $v\in V$. Since
$E\big[f(\Phi_{0,t+s}(y_0))|\Phi_{0,s}(y_0)\big]=g_{s,t,f}(\Phi_{0,s}(y_0))$ we obtain
\begin{align*}
E\big[f(\Phi_{0,t+s}(y_0))|\mathcal{F}_s\big]=E\big[f(\Phi_{0,t+s}(y_0))|\Phi_{0,s}(y_0)\big],
\end{align*}
which completes the proof of Part (b).
\end{proof}

\noindent\textbf{Acknowledgments:} the authors would like to thank 
Tomasz Kosmala for proofreading. Umesh Kumar thanks his home institution
 Rajdhani College, University of Delhi, New Delhi - 110015, INDIA for granting him leave for his PhD studies. 


\begin{thebibliography}{10}

\bibitem{app_martingale}
D.~Applebaum.
\newblock Martingale-valued measures, {O}rnstein-{U}hlenbeck processes with
  jumps and operator self-decomposability in {H}ilbert space.
\newblock In {\em In memoriam {P}aul-{A}ndr\'e {M}eyer: {S}\'eminaire de
  {P}robabilit\'es {XXXIX}}, pages 171--196. Berlin: Springer, 2006.

\bibitem{app}
D.~Applebaum and M.~Riedle.
\newblock Cylindrical {L}\'evy processes in {B}anach spaces.
\newblock {\em Proc. Lond. Math. Soc.}, 101(3):697--726, 2010.

\bibitem{bourbaki}
N.~Bourbaki.
\newblock {\em Functions of a real variable}.
\newblock Berlin: Springer-Verlag, 2004.

\bibitem{Brzetal}
Z.~Brze\'zniak, B.~Goldys, P.~Imkeller, S.~Peszat, E.~Priola, and J.~Zabczyk.
\newblock Time irregularity of generalized {O}rnstein-{U}hlenbeck processes.
\newblock {\em C. R. Math. Acad. Sci. Paris}, 348(5-6):273--276, 2010.

\bibitem{brzezniak_neerven}
Z.~Brze\'zniak and J.~M. A.~M. van Neerven.
\newblock Stochastic convolution in separable {B}anach spaces and the
  stochastic linear {C}auchy problem.
\newblock {\em Studia Math.}, 143(1):43--74, 2000.

\bibitem{brzezniak_zabczyk}
Z.~Brze\'zniak and J.~Zabczyk.
\newblock Regularity of {O}rnstein-{U}hlenbeck processes driven by a {L}\'evy
  white noise.
\newblock {\em Potential Anal.}, 32(2):153--188, 2010.

\bibitem{michalik}
A.~Chojnowska-Michalik.
\newblock On processes of {O}rnstein-{U}hlenbeck type in {H}ilbert space.
\newblock {\em Stochastics}, 21(3):251--286, 1987.

\bibitem{daprato_zab}
G.~Da~Prato and J.~Zabczyk.
\newblock {\em Stochastic equations in infinite dimensions}.
\newblock Cambridge: Cambridge University Press, 2014.

\bibitem{Dieudonne}
J.~Dieudonn\'e.
\newblock {\em Foundations of modern analysis}.
\newblock New York: Academic Press, 1969.

\bibitem{Dunford_Schwartz}
N.~Dunford and J.~Schwartz.
\newblock {\em Linear operators. {P}art {I}}.
\newblock John Wiley \& Sons, Inc., New York, 1988.

\bibitem{teichman}
D.~Filipovi\'c, S.~Tappe, and J.~Teichmann.
\newblock Jump-diffusions in {H}ilbert spaces: existence, stability and
  numerics.
\newblock {\em Stochastics}, 82(5):475--520, 2010.

\bibitem{jaku1}
A.~Jakubowski.
\newblock Tightness criterion for random measures with application to the
  principle of conditioning in {H}ilbert spaces.
\newblock {\em Probab. Math. Stat.}, 9(1):95--114, 1998.

\bibitem{adam}
A.~Jakubowski and M.~Riedle.
\newblock Stochastic integration with respect to cylindrical {L}{\'e}vy
  processes.
\newblock {\em Annals of Probability}, 45:4273--4306, 2017.

\bibitem{linde}
W.~Linde.
\newblock {\em Probability in Banach spaces--stable and infinitely divisible
  distributions}.
\newblock Chichester: John Wiley \& Sons, Ltd., 1986.

\bibitem{LiuZhai}
Y.~Liu and J.~Zhai.
\newblock A note on time regularity of generalized {O}rnstein-{U}hlenbeck
  processes with cylindrical stable noise.
\newblock {\em C. R. Math. Acad. Sci. Paris}, 350(1-2):97--100, 2012.

\bibitem{LiuZhai16}
Y.~Liu and J.~Zhai.
\newblock Time regularity of generalized {O}rnstein-{U}hlenbeck processes with
  {L}{\'e}vy noises in {H}ilbert spaces.
\newblock {\em J. Theoret. Probab.}, 29(3):843--866, 2016.

\bibitem{partha}
K.~R. Parthasarathy.
\newblock {\em Probability measures on metric spaces}.
\newblock Providence, RI: AMS Chelsea Publishing, 2005.

\bibitem{peszat_zabczyk}
S.~Peszat and J.~Zabczyk.
\newblock {\em Stochastic partial differential equations with L\'evy noise. An
  evolution equation approach}.
\newblock Cambridge: Cambridge University Press, 2007.

\bibitem{PeszatZab12}
S.~Peszat and J.~Zabczyk.
\newblock Time regularity of solutions to linear equations with {L}\'evy noise
  in infinite dimensions.
\newblock {\em Stochastic Process. Appl.}, 123(3):719--751, 2013.

\bibitem{priola_zabczyk}
E.~Priola and J.~Zabczyk.
\newblock Structural properties of semilinear {SPDE}s driven by cylindrical
  stable processes.
\newblock {\em Probab. Theory Related Fields}, 149(1-2):97--137, 2011.

\bibitem{Riedle_alpha_stable}
M.~Riedle.
\newblock Stable cylindrical {L}{\'e}vy processes and the stochastic {C}auchy
  problem.
\newblock {\em preprint}.
\newblock \url{https://kclpure.kcl.ac.uk/portal/files/71677012/stable3.pdf}.

\bibitem{Riedle_infinitely}
M.~Riedle.
\newblock Infinitely divisible cylindrical measures on {B}anach spaces.
\newblock {\em Studia Math.}, 207(3):235--256, 2011.

\bibitem{OU}
M.~Riedle.
\newblock Ornstein-{U}hlenbeck processes driven by cylindrical {L}\'evy
  processes.
\newblock {\em Potential Anal.}, 42(4):809--838, 2015.

\bibitem{neerven}
J.~M. A.~M. van Neerven.
\newblock {\em Stochastic Evolutions Equations}.
\newblock ISEM Lecture Notes, 2007.

\bibitem{neerven_veraar}
J.~M. A.~M. van Neerven and M.~Veraar.
\newblock On the stochastic {F}ubini theorem in infinite dimensions.
\newblock In {\em Stochastic partial differential equations and
  applications---{VII}}, pages 323--336. Boca Raton, FL: Chapman \& Hall/CRC,
  2006.

\bibitem{neerven_weis}
J.~M. A.~M. van Neerven and L.~Weis.
\newblock Stochastic integration of functions with values in a {B}anach space.
\newblock {\em Studia Math.}, 166(2):131--170, 2005.

\end{thebibliography}

\end{document}